\documentclass[12pt]{amsart}
\usepackage{amsfonts,amssymb,eucal}

\usepackage{color}
\usepackage{amsthm}
\usepackage{amsmath}
\usepackage{amscd}
 \usepackage{latexsym} 

\usepackage{bbold,mathbbol,bbm}
\input{cyracc.def}

\numberwithin{equation}{section}
\newcommand{\tr}{\operatorname{tr}}

\newcommand{\mb}[1]{{\mbox{\boldmath{$#1$}}}}
\newcommand{\mc}[1]{{\mathcal{#1}}}
\newcommand{\got}[1]{{\mathfrak{#1}}}

\newcommand{\pa}{\partial}

\newcommand{\R}{\ensuremath{\mathbb{R}}}
\newcommand{\K}{\ensuremath{\mathbb{K}}}
\newcommand{\C}{\ensuremath{\mathbb{C}}}
\newcommand{\N}{\ensuremath{\mathbb{N}}}

\newcommand{\db}[1]{{\mathbb{#1}}}
\newcommand{\Z}{\ensuremath{\mathbb{Z}}}

\newcommand{\Sp}{\ensuremath{{\mbox{\rm{Sp}}(n,\R)}}}
\newcommand{\SPn}{\ensuremath{{\mbox{\rm{Sp}}(n+1,\R)}}}
\newcommand{\oo}{\ensuremath{\mathbb{1}}}

\newcommand{\UU}{\ensuremath{\mathbb{0}}}
\newcommand{\UUN}{\ensuremath{\mathbb{0}_n}}
\newcommand{\un}{{\oo_n}}
\def\ii{\operatorname{i}}

\newtheorem{Theorem}{Theorem}
\newcommand{\dd}{\operatorname{d}}
\newcommand{\Ka}{K\"ahler}
\newtheorem{Proposition}{Proposition}
\newtheorem{lemma}{Lemma}

\theoremstyle{definition}
\newtheorem{Remark}[lemma]{Remark}

\newcommand{\mr}[1]{{\mathrm{#1}}}

\allowdisplaybreaks

\begin{document}

\title{Invariant metric on the extended Siegel-Jacobi  upper  half space}

\author{Stefan  Berceanu}
\address[Stefan  Berceanu]{Horia Hulubei National
 Institute for Physics and Nuclear Engineering\\
        Department of Theoretical Physics\\
 P.O.B. MG-6, 077125 Magurele, Romania}
\email{ Berceanu@theory.nipne.ro}

\begin{abstract} The real Jacobi group  $G^J_n(\mathbb{R})$,  defined as the semidirect product of the
Heisenberg group ${\rm H}_n(\R)$ with the symplectic group ${\mr {Sp}}(n,\mathbb{R})$,  admits a matrix  
embedding  in $\text{Sp}(n+1,\mathbb{R})$. The modified 
pre-Iwasawa decomposition of $\rm{Sp}(n,\mathbb{R})$  allows us to
 introduce    a convenient  coordinatization   $S_n$ of
 $G^J_n(\mathbb{R})$,  which for  $G^J_1(\mathbb{R})$  coincides with the  $S$-coordinates.
Invariant one-forms on $G^J_n(\mathbb{R})$ are determined.  
The formula of the  4-parameter invariant metric on 
$G^J_1(\R)$  obtained as sum of squares of 6  invariant one-forms is 
extended to $G^J_n(\R)$,   $n\in\mathbb{N}$.  We obtain a three
parameter   invariant metric on  the  extended
  Siegel-Jacobi upper  half space 
  $\tilde{\mathcal{X}}^J_n\approx\mathcal{X}^J_n\times \mathbb{R}$ by 
  adding the square of an invariant one-form to the  two-parameter  balanced metric
  on  the Siegel-Jacobi upper
half space
$ {\mathcal{X}}^J_n =\frac{G^J_n(\mathbb{R})}{\mr{U}(n)\times\mathbb{R}}$.
\end{abstract}
\maketitle

\tableofcontents

\section{Introduction}

The real Jacobi group \cite{bs,ez,tak,ZIEG}
of  degree $n$ is defined as $G^J_n(\R):=\mr{H}_n(\R)\rtimes\Sp$, where
$\mr{H}_n(\R)$ denotes the   real Heisenberg group.
The Siegel-Jacobi  upper half space is the  $G^J_n(\R)$-homogeneous
manifold 
$\mc{X}^J_n:=\frac{G^J_n(\R)}{\rm{U}(n)\times \R}\approx \mc{X}_n\times
\R^{2n}$ \cite{SB15,SB19}, \cite{yang}-\cite{Y08}, where  $\mc{X}_n$ denotes the Siegel upper half space realized as 
  $ \frac{\Sp}{\rm{U}(n)}$ \cite[p 398]{helg}.

  The Jacobi
group 
$G^J_n:=\mr{H}_n\rtimes{\rm Sp}(n,\R)_{\C}$,
where ${\rm Sp}(n,\R)_{\C}:= {\rm Sp}(n,\C)\cap {\rm U}(n,n)$
\cite{multi,coord} is  also studied in Mathematics,    Mathematical
Physics and Theoretical Physics,   together  with the 
 $G^J_n$-homogeneous Siegel-Jacobi ball $\mc{D}^J_n\approx
\C^n\times\mc{D}_n$  \cite{multi},
where $\mc{D}_n$ denotes the
Siegel ball realized as $\frac{{\rm Sp}(n,\C)}{{\rm U}(n)}$
\cite[p 399]{helg}.  

 It
is well known that $G^J_n(\R)$, ${\rm Sp}(n,\R)$, $\mr{H}_n(\R)$,
$\mc{X}^J_n$
 and $\mc{X}_n$ are isomorphic with $G^J_n$, ${\rm Sp}(n,\R)_{\C}$, $\mr{H}_n$,
$\mc{D}^J_n$, 
respectively  $\mc{D}_n$, see  \cite{multi}-\cite{FC}, \cite{gem, bs,ez,yang,Y07}.

The dimensions of the enumerated manifolds are: $\dim{\Sp}=2n^2+n$,
$\dim{\mr{H}_n(\R)}$$=2n+1$,
  $\dim{G^J_n(\R)}=(2n+1)(n+1)$,
  $\dim{\mr{U}(n)}=n$,
  $\dim{\mc{X}^J_n}=n(n+3)$,
    $\dim{\tilde{\mc{X}}^J_n}=n(n+3)+1$, $\dim{\mc{X}_n}=n(n+1)$.

The Jacobi group, as  a unimodular, non-reductive, algebraic group of Harish-Chandra
type \cite{gem,LEE03}, \cite{SA71}-\cite{SA80}, also  a coherent state
(CS) type group \cite{SB03,lis2,lis,mosc,mv,neeb} is an interesting object in
Mathematics   \cite{coord,SB15}.
$\mc{D}^J_n$ is a
 partially bounded domain, non-symmetric,  a  Lu Qi-Keng manifold,  a
 projectively induced quantizable \Ka~manifold \cite{SB15}, \cite{yang,Y07}.

 The Jacobi group has many applications in several branches of Physics: quantum mechanics,
geometric quantization, nuclear structure, signal processing, quantum optics,
in particular squeezed states and quantum teleportation, see references in \cite{SB19}.
The Jacobi group was known  to 
physicists under other names as  Hagen
  \cite{hagen},  Schr\"odinger \cite{ni},    or  Weyl-symplectic group
 \cite{kbw1}. The Jacobi
group is  responsible  for the   squeezed states  \cite{ho,stol,lu,yu, ken} in quantum optics
\cite{ali,dod,dr,mandel,siv}. 

 The Jacobi
group was investigated  in several  publications
\cite{holl}-\cite{coord}, \cite{SB15}  with  Perelomov's CS method
 \cite{perG} based on \Ka~
manifolds
\cite{ber73}-\cite{berezin}, \cite{lis2,lis,mosc,mv,neeb}
  associated  to
$G^J_n$, determining the balanced metric \cite{don}. Berezin's
quantization \cite{ber73}-\cite{berezin}, \cite{Cah,cah,eng,raw}
related to the Jacobi group has been also  investigated
\cite{SB15,ca1,ca2}. But the  CS method is
applicable only to \Ka~ manifolds. Because  it is
desirable  to
impose  the invariance of    metric also on manifolds of odd dimension, 
the CS  method  must  be abandoned. 

 In this paper  we introduce   an
  odd dimensional manifold, called   
  extended Siegel-Jacobi upper half space of order n, 
  $\tilde{\mc{X}}^J_n:=\frac{G^J_n(\R)}{\mr{U}(n)}\approx
  \mc{X}^J_n\times\R $, a generalization of the 5-dimensional
  Siegel-Jacobi upper half-plane
  ${\tilde{\mc{X}}^J_1}=\frac{G^J_1(\R)}{\text{SO}(2)}\approx\mc{X}_1\times\R^3$
  considered  in \cite{SB20,SB19}.
  Because   the Jacobi group governs
 the squeezed  states \cite{SB81,SB82}, we are  expecting  that   the manifold
 $\tilde{\mc{X}}^J_n$ to have applications 
  in quantum optics. We recall  that the squeezed states are a particular class of ``minimum
uncertainty states'' (MUS) \cite{mo} and  that ``Gaussian pure states''
(``Gaussons'') \cite{si} are more general MUSs.

The invariant metrics on homogeneous manifolds associated to
the  real  Jacobi group  $G^J_1(\R)$ were obtained in \cite{SB20,SB19},
applying  Cartan's  moving frame method
\cite{cart4,cart5,ev}. We have
determined a 3-parameter  invariant metric on the extended Siegel-Jacobi
upper half-plane   
\cite{SB20,SB19}. To get  the
invariant metric on   ${\tilde{\mc{X}}^J_1}$, we have determined the
invariant one-forms $\lambda_1,\dots,\lambda_6$ on $G^J_1(\R)$. Then
we have calculated  the invariant vector fields $L^j$ 
verifying the relations $<\lambda_i|L^j>=\delta_{ij}$,
$i,j=1,\dots,6$,  such that $L^j$
are orthonormal with respect to the  4-parameter invariant 
  metric   
$\dd s^2_{G^J_1(\R)}$ expressed  in the $S$-coordinates $(x,y,\theta,p,q,\kappa)$
 \cite[p  10]{bs},
where $\theta \in [0,2\pi)$ and the other $S$-coordinates are in $\R$.

In the present paper we apply to $G^J_n(\R)$, $n\in \N$,   the method
applied in \cite{SB15} to  $G^J_1(\R)$. Firstly  we   determine the
invariant
one-forms on $G^J_n(\R)$.  If a point  $g \in G^J_n(\R)$ is
parametrized by  the coordinates $(M,X,\kappa)$, 
where $M\in{\rm Sp}(n,\R)$,
$X:=(\lambda,\mu)\in M(1,2n,\R), \kappa\in\R$, and $(p,q)=XM^{-1}$,
then we have the
  following representation of the real Jacobi group embedded  in \SPn ~\cite{Yg93,ZIEG}
  \begin{equation}\label{ggmin1}
  g =\left(\begin{array}{cccc}a & \UU_{n1} & b& q^t\\
              \lambda & 1& \mu &\kappa\\
              c & \UU_{n1} &  d & -p^t\\
              \UU_{1n} &  0 & \UU_{1n}&
                                        1\end{array}\right),M=\left(\begin{array}{cccc}a&b\\c&d\end{array}\right)\in\Sp.
 \end{equation}
       In this paper  we parametrize the group  $G^J_n(\R)$ with  a
       system  of coordinates 
$(x,y,X,Y,$\newline $p,q,\kappa)$,
where $x+\ii y\in \mc{X}_n$,  $X+\ii Y\in { \mr U}(n)$, while $(p,q,\kappa)$ characterize the Heisenberg group
${\rm H}_n(\R)$. This  system of coordinates, denoted $S_n$, $n\in \N$, coincides
for $n=1$ with the  $S$-coordinates of
$G^J_1(\R)$ \cite[p 10]{bs}.
 The main ingredient of  the $S_n$-parametrization of $G^J_n(\R)$ is the modified pre-Iwasawa decomposition of the
symplectic group $\Sp$, inspired by \cite{arw,goss}. We obtain a
4-parameter 
invariant metric on $G^J_n(\R)$, which in the case  $n=1$ coincides
with the 4-parameter  invariant  metric determined in \cite{SB19}. However, the
explicit expressions  for the metrics in Proposition \ref{PRm}
obtained from  the
invariant one forms on $G^J_n(\R)$ are quite
complicated, so in order to obtain the invariant metric on the odd
dimensional extended Siegel-Jacobi space $\tilde{\mc{X}}^J_n$  we just
add  the
square of an invariant one-form attached to $\kappa$ to the 2-parameter balanced metric of the Siegel-Jacobi upper half space
obtained via the CS method in \cite{coord,SB15}.

The paper is organized as follows. Section \ref{HL} summarizes the
embedding of the 
Heisenberg group ${\rm H}_n(\R)$ in \SPn.  Section \ref{SPK} describes    the
symplectic group. The pre-Iwasawa 
decomposition is  introduced in Lemma
\ref{R313}, while Lemma \ref{L5} shows that the modified pre-Iwasawa
decomposition is compatible with the linear fractional action of \Sp~
on  $\mc{X}_n$. Section \ref{GHHN} considers the
real Jacobi group $G^J_n(\R)$. The embedding of $G^J_n(\R)$ in \SPn~is
described in Remark \ref{L44}. After choosing  a base of the Lie algebra
$\got{g}^J_n(\R)$  which in particular for $n=1$ coincides with that
in \cite{SB19}, Lemma \ref{actM} describes the action
of the Jacobi group on the homogeneous manifolds $\mc{X}^J_n$ and 
$\tilde{\mc{X}}^J_n$.   In Section \ref{FVF} are calculated  the fundamental
vector fields  (FVF)  associated to the generators of the Jacobi
group on $\mc{X}^J_n$ and 
$\tilde{\mc{X}}^J_n$.  In  Section \ref{IOF} are  obtained the
invariant one-forms on $G^J_n(\R)$ in the $S_n$-coordinates, see Lemma
 \ref{L9}  and \eqref{LFGHL}. The difficulties to calculate the
invariant vector fields once the invariant one-forms are known are
exemplified in Section \ref{difi}. Proposition \ref{PRm} expresses the
4-parameter invariant metric on $G^J_n(\R)$. Proposition \ref{L11}, an extension to
 $n\in\N$ of \cite[Proposition1]{SB20}, expresses the \Ka~two-form
on $\mc{X}_n^J$ in several types of  variables. Remark \ref{REM11}
gives a CS-meaning to the  $S_n$-parameters $p,q$ describing
$G^J_n(\R)$. The invariant metric on the
odd dimensional manifold $\tilde{\mc{X}}^J_n$ is given in Theorem  \ref{PR2}. Finally,
other parametrizations  of the Jacobi algebra $\got{g}^J_n(\R)$  are recalled  in
\S~\ref{AP7}, while Section \ref{AP8} summarizes the method of calculating
the differential of square root of a  symmetric matrix.

To conclude, the new results of this  paper are contained 
in Lemma  \ref{R313}, Lemma \ref{L5}, parts of Lemma \ref{actM}, the
base \eqref{alg2} of $\got{g}^J_n(\R)$, Lemma \ref{L9},
Propositions \ref{FDVF} -- \ref{L11} and Remark \ref{REM11}. The main
result of the present investigation is   stated  in Theorem
\ref{PR2}.

\vspace{1cm}
        
{\bf Notation}   We denote by $\mathbb{R}$, $\mathbb{C}$, $\mathbb{Z}$,
and $\mathbb{N}$ the field of real numbers, the field of complex numbers,
the ring of integers, and the set of positive  integers, respectively.
We denote by $\ii$  the imaginary unit
$\sqrt{-1}$, and the complex 
conjugate of $z$ by $\bar{z}$. We denote the set of  $m\times n$ matrice with entries in the field
$\db{F}$ as $M(m,n;\db{F})$ and if $n=m$  we write
$M(n,\db{F})$. $M(n,\db{F})$ for $\db{F}$ equal with $\R$ or $\C$ is
denoted by $M(n)$.   We denote
the transpose  (the Hermitian conjugate) of the matrix $A$ by $A^t$,
(respectively $A^{\dagger}$). $\un$ denotes the identity matrix of
$M(n,\db{F})$, while  $\UU_{nm}\in M(n,m,\db{F})$ denotes   the matrix with   all
elements   zero and $\UU_n$ means  $\UU_{nn}$.                                       
   $E_p\in M(1,n,\R)$ denotes the matrix with 
  1 on the position $p$,  $(E_p)_i=\delta_{pi}$ and similarly for
  $E_q$, $p,q =1,\dots, n$. $E_{ij}$ denotes the square matrix with entry 1 at the
 intersection of  the $i$th row  with the $j$th column, 
 $(E_{ij})_{kl}=\delta_{ik}\delta_{jl}$, and
 $E_{ij}E_{kl}=\delta_{jk}E_{il}$. 
When the dimension of a submatrix  of a block matrix is not
evident, the   subindices  $pq$
 specify  that the
 respective submatrix  is in $ M(p,q,\R)$. 
We denote by $\dd$  the differential. We use Einstein convention  that repeated indices are
implicitly summed over. We denote  by $\mr{dg}(a_1,\dots,a_n)$
the diagonal matrix which has on diagonal 
$a_1,\dots,a_n$. 
We denote by $<\lambda|L>$ the pairing of the  one-form $\lambda$ with
the vector field $L$. We consider a complex separable
Hilbert space $\got{H}$ endowed with a scalar product $(\cdot,\cdot)$ which is antilinear in the first argument,
$(\lambda x, y) =\bar{\lambda}(x,y)$
$x,y\in\got{H}$, $\lambda\in \C\setminus  0$. 
If $\pi$ is a representation of a Lie group $G$ on the Hilbert space
$\got{H}$ and $\got{g}$ is the Lie algebra of $G$, we
denote $ \mb{X} := \dd \pi(X)$ for $X\in \got{g}$.

\section{The Heisenberg group ${\rm H}_n(\R)$ as subgroup of \SPn}\label{HL}
The real Heisenberg group 
${\rm H}_n(\R)$, parametrized by $(\lambda,\mu,\kappa)$,
$\lambda,\mu\in M(1,n,\R)$, $\kappa\in\R$, has the composition law
\cite{gem,LEE03,SA71C,SA80,Yg93,ZIEG} 
\begin{equation}\label{CLH}
  (\lambda,\mu,\kappa)\times(\lambda',\mu',\kappa')=
    (\lambda+\lambda',\mu+\mu',\kappa+\kappa'+\lambda\mu'^t-\mu\lambda'^t).
    \end{equation}
${\rm H}_n(\R)$ is a particular case of the Heisenberg group $H^{(n,m)}_{\R}$ for $m=1$, see
\cite{yang} and \cite{Y10}.

If $g\in {\rm H}_n(\R)$, we represent it   \cite{Yg93,ZIEG} and its  inverse  embedded in \SPn~ as
\begin{equation}\label{ggminH}g=\left(\begin{array}{cccc}
1 & 0 &0 &\mu^t\\ \lambda & 1&\mu &\kappa\\
                                        0 & 0& 1 & -\lambda^t\\
                                        0& 0& 0& 1\end{array}\right),~
 g^{-1}=\left(\begin{array}{cccc} 1 & 0 &0 & -\mu^t\\
    -\lambda &1-&\mu& - \kappa\\
    0 & 0& 1 & \lambda^t\\
    0& 0& 0 & 1\end{array}\right),\end{equation}
see also notation in \eqref{ggmin1} and  Lemma  \ref{L44}.

If the generators $P_p,Q_q$,  $p,q=1,\dots,n$, $R$,  of the Heisenberg  group are defined
in \eqref{algp}, see also the last three equations in \eqref{alg1}
and Lemma  \ref{L44}, 
\begin{subequations}\label{algp}
  \begin{align}
\label{P1}    P_p & =\left(\begin{array}{cccc} 0 & 0 & 0& 0\\
                      E_{p}& 0& 0& 0\\
                       0 & 0& 0& -E^t_{p}\\
                                                           0& 0& 0&
                                                                    0\\ \end{array}\right),~p=1,\dots,n\\
\label{Q1}      Q_q & =\left(\begin{array}{cccc} 0 & 0 & 0& E^t_{q}\\
                      0 & 0& E_{q}& 0\\
                       0 & 0& 0&  0\\
                                                           0& 0& 0&
                                                                    0 \end{array}\right),~q=1,\dots,n,\\
                                                                  \label{R11}  R & =\left(\begin{array}{cccc} 0 & 0 & 0& 0\\
                      0& 0& 0& 1\\
                       0 & 0& 0& 0\\
                                                           0& 0& 0&
                                                                    0\\ \end{array}\right),
  \end{align}
  \end{subequations}
then \begin{equation}\label{ggmdg}g^{-1}\dd g=P_p\lambda^p+Q_q\lambda^q+R\lambda^r.\end{equation}
            With \eqref{ggminH} and \eqref{ggmdg},
            the left invariant one-forms on ${\rm H}_n(\R)$ are
            \begin{equation}\label{lplqlr}
                     \lambda^p = \dd \lambda_p,\quad
                     \lambda^q  =\dd \mu_q,\quad 
                     \lambda^r = \dd \kappa -\lambda \dd \mu^t+\mu
                                 \dd \lambda^t.
                   \end{equation}
                   The  left action  of the Heisenberg group on itself
                   is obtained from \eqref{CLH}
                   $$\exp(\lambda P+\mu^tQ+\kappa
R)(\lambda_0,\mu_0,\kappa_0)=(\lambda+\lambda_0,\mu+\mu_0,
\kappa+\kappa_0+\lambda\mu_0^t-\mu\lambda_0^t).$$
The left invariant metric on the Heisenberg group is
  \[
    g^L(\lambda,\mu,\kappa)=\dd\lambda^2+\dd \mu^2+(\dd \kappa
    -\lambda \dd \mu^t+\mu \dd \lambda^t)^2.
  \]

The fundamental vector fields,  see  \cite[p. 121,  Ch II
  \S~3]{helg},  \cite[p. 42]{kn1}, or \cite[\S ~6.1,  v1]{SB19}, on the Heisenberg group ${\rm H}_n(\R)$ are
\[
    P^*  =\frac{\pa}{\pa_{\lambda}}+\mu^t\frac{\pa}{\pa_{\kappa}},\quad 
    Q^*  =\frac{\pa}{\pa_{\mu}}-\lambda\frac{\pa}{\pa_{\kappa}},\quad 
    R^*=\frac{\pa}{\pa_{\kappa}}.
    \]
    See also \eqref{EQQ6}.

    \section{The symplectic group \Sp}\label{SPK}
    \subsection{Basics}

The group $\text{Sp}(n,\K)$ admits a  matrix realization in 
$M\in\text{M}(2n,\K)$, where $\K$ is $\R$ or $\C$,
verifying the relation
\begin{equation}\label{XXX}
 M^tJ_nM=J_n, \quad J_n:= \left(\begin{array}{cc} \UUN & \un\\ -\un &
                                                                      \UUN \end{array}\right).
\end{equation}

If there is no possibility of confusion, we denote $J_n$ just with
$J$.

Let us consider a matrix 
\begin{equation}\label{Real}
M= \left(\begin{array}{cc} a& b\\ c&  d\end{array}\right),\quad
a,b,c,d\in M(n,\R).\end{equation}
It is easy to prove  \cite{frei}-\cite{fol},  \cite{sieg} that
\begin{Remark}
  a) If $M\in\mathrm{Sp}(n,\R)$, then $M$ is similar with
  $M^t$ and $M^{-1}$ and also $J\in\Sp$.
  
b) If $ M\in \Sp$ is as in \eqref{Real},  then the
matrices $a,b,c,d$ in \eqref{Real}
  verify  the sets of  equivalent conditions 
\begin{subequations}\label{simplecticR}
\begin{align}
{\mbox{~~}} & ab^t- ba^t = \UUN,~  ad^t-bc^t =\un, ~ cd^t-dc^t=\UUN; \label{33a}\\
 {\mbox{~~}} &a^tc-c^ta=\UUN, ~ a^td-c^tb=\un,~ b^td-d^tb=\UUN.
\end{align}
\end{subequations}
c) If  $M\in\Sp$ has the form  \eqref{Real}, then   \begin{equation}\label{MINV}
  M^{-1}=\left(\begin{matrix} d^t & -b^t\\ -c^t &
      a^t\end{matrix}\right) .\end{equation}

d) The matrices in   $\mathrm{Sp}(n,\R)$ have the determinant
1.

e) The following  subsets of  $\mr{GL}(2n,\R)$
are subgroups of \Sp
\begin{equation*}N \!=\! \left\{\!\left(\begin{array}{cc}\un & A\\ \mathbb{O}_n &
                                                  \un\end{array}\right): A\!=\!A^t\right\},
                                              \tilde{N}=\left\{\left(\begin{array}{cc}\un
                                                               & \mathbb{O}_n
                                                               \\ B &
                                                                      \un
                                                                      \end{array}\right)
                                                                    : B\!=\!B^t
                                                                  \right\},
                                                                 \end{equation*}
                                                                \begin{equation*}
                                                                  D \!=\!
                                                                  \left\{\left(\begin{array}{cc}C
                                                                           &
                                                                           \mathbb{O}_n\\
                                                                 \mathbb{O}_n
                                                                           &
                                                                             (C^t)^{-1} 
                                                \end{array} \right): C\in\mr{GL}(n,\R)\right\}.
                                            \end{equation*}
\Sp~ is generated by $D\cup\tilde{N}\cup \{J\}$
                                            and $D\cup{N}\cup \{J\}$.
\end{Remark}
Using  \eqref{MINV}  it can be shown that the matrix
$\mc{M}\in\mathrm{Sp}(n,\R)\cap\mathrm{O}_{2n}$ has  the expression
\begin{equation}\label{msimor}\mc{M}=\left(\begin{array}{cc} X & Y \\ -Y & X\end{array}\right),
~~ X^tX+ Y^tY=XX^t+YY^t=\un,~~ X^tY=Y^tX,~~ YX^t=XY^t.\end{equation}
If $\mc{M}\in M(2n,\R)$  has the properties  \eqref{msimor},
let
 \begin{equation}\label{MPR} 
\mc{M}':=X+\ii Y\in M(n,\C),
\end{equation}
and 
\begin{Remark}\label{REM2} The correspondence $\mc{M}\rightarrow \mc{M}'$ of \eqref{msimor}
  with \eqref{MPR} 
  is a  group 
  isomorphism  
  \[
    \mathrm{Sp}(n,\R)\cap\mathrm{O}_{2n}\approx\text{U}(n).
  \]
\end{Remark}
\subsection{The real symplectic algebra $\got{sp}(n,\R)$}

The real symplectic Lie algebra $\got{g}=\got{sp}(n, \R )$ is a real form
of the simple Lie algebra $ \got{sp}(n, \C )$ of type $\got{c}_n$
and $ X\in \got{sp}(n, \R )$ if $X^t J+JX=0$, or equivalently
\begin{equation}\label{xXx}
X=\left( \begin{matrix} a & b \\c & -a^t\end{matrix}\right),\quad b=b^t,\quad c=c^t,
\end{equation}
 where $a,b,c\in M(n,\R )$, and similarly for
 $ \got{sp}(n, \C )$.
 
We write an element $X$ (\ref{xXx})  as
\[
X=\sum_{i,j} a_{ij}H_{ij}+2\sum_{i<j}(b_{ij}F_{ij}+ c_{ij}G_{ij})+
\sum_{i=j}(b_{ij}F_{ij}+ c_{ij}G_{ij}),\quad 1\le i,j\le n, 
\]
\begin{equation}\label{HFV}
  H_{ij}:=\left(\begin{array}{cc} E_{ij} & \UUN \\ \UUN & -
    E_{ji}\end{array}\right), 2F_{ij}:= \left(\begin{array}{cc} \UUN &   E_{ij}+E_{ji}\\ \UUN &
   \UUN  \end{array}\right) ; 
 2G_{ij}:= \left(\begin{array}{cc} \UUN &  \UUN\\  E_{ij}+E_{ji} & \UUN
\end{array}\right). \end{equation}

 In the matrix realization (\ref{xXx}),  the real algebra $\got{sp}(n,\R)$  has the $2n^2+n$ generators
\begin{equation}\label{HFG}
 H_{ij}, \quad F_{ij}, \quad G_{ij}, \quad
1\le i\le j\le n.
\end{equation}

\subsection{ $\mc{X}_n$ as Hermitian symmetric space} 

We briefly recall some well
known facts about Hermitian symmetric
spaces \cite{multi,helg,wolf}. We use the notation

\smallskip
$X_n$:\quad Hermitian symmetric space of noncompact type,
$X_n= G_0/K$;

$X_c$:\quad compact dual of $X_n$, $X_c= G_c/K$; 

$G_0$:\quad real Hermitian group;

$G=G_0^{\C}$:\quad  the complexification
of $G_0$;

$P$:\quad
a parabolic subgroup of $G$;

$K$:\quad maximal compact subgroup of $G_0$;

$G_c$:\quad compact real form of $G$.

The compact manifold $X_c$ of $\frac{n(n+1)}{2}$-complex dimension
has a complex structure inherited from the identification of  $X_c$
with $ G/P $. The group $G_c$ acts transitively on $X_c$  with
isotropy group $
K= G_0\cap P= G_c\cap P$.

$X_n =G_0/K= G_n(x_0)$ is open in $X_c$, where $x_0$ is a base
point of $G$ corresponding to $K$. If $\{e_1,\dots,e_{2n}\}$ is a base of $\C^{2n}$, 
in our case we take  $x_0=e_{n+1}\wedge\cdots\wedge e_{2n}\in X_n$ as base point, and 
$ G_0= \operatorname{Sp}(n,\R )\approx \operatorname{Sp}(n,\R )_{\C}$.

$X_c$ includes $X_n$ under
Borel embedding $X_n\subset X_c: gK\mapsto gP$, $g\in G_0$.

The hermitian form on $\C^{2n}$
$$<u,v>=-\sum_{j=1}^nu^j\bar{v}^j+\sum_{k=1}^nu^{n+k}\bar{v}^{n+k}$$
specifies the indefinite unitary group $\text{U}(n,n)$, hence the transformation  group  $\operatorname{Sp}(n,\R )_{\C}$ acting on $X_n$ 
$$ G:=\operatorname{Sp} (n,\C
),~ G_c:= \operatorname{Sp} (n)= \operatorname{Sp} (n, \C )\cap
\operatorname{U}(2n)\subset \operatorname{SU}(2n),~ K:=
\operatorname{U}(n).$$
We have also
\[
  P:=\left\{g\in G:g(x_0)=x_0\right\}= \left\{\left(\begin{matrix} a & \UUN\\
        c & d \end{matrix}\right): a^t c=c^ta, \ a^t d = \un \right\}.
\]
Let us consider also
\[
\got{m}^+:= \left\{\left(\begin{matrix}\UUN & b \\ \UUN & \UUN
    \end{matrix}\right): b^t=b \right\}, b\in M(n,\C).
\]
Then
\begin{equation}\label{HCE}
W\mapsto \hat{W}\!=\!\left(\begin{matrix}\UUN & W\\ \UUN & \UUN\end{matrix}
\right),\xi(W)\!=\!( \exp\hat{W})x_0\!=\! v_1\wedge\cdots \wedge v_n,
(v_1,\cdots,v_n)\!=\!\left(\begin{array}{c} W\\\un\end{array}\right),
\end{equation}
and   $\xi$ maps the symmetric $n\times n$ matrices $W$
 of $\got{m}^+$
such that
$\un-W\bar{W}>0$ onto a dense open subset of $X_c$
that contains $X_n$. 

$X_n$ is a Hermitian symmetric space of type CI (cf.~Table V, p.~
518, in \cite{helg}), identified with the  symmetric bounded
domain of type II, $\got{R}_{II}$ in Hua's notation \cite{hua}.

 Let us denote by $\mc{X}_n$ the set
\begin{equation}\label{SGUP}\mc{X}_n:=\{v\in M(n,\C)| v=s+\ii  r, s,
r\in M(n,\R), 
r>0,
s^t=s; r^t=r\} . \end{equation}
\begin{Remark}\label{direct}The action \eqref{conf1}  of  $\mathrm{Sp}(n,\R)$ on the
  Siegel upper half space  $\mc{X}_n$
\begin{equation}\label{conf1}
v_1  = M(v)=(av+b)(cv+d)^{-1}=(vc^t+d^t)^{-1}(va^t+b^t).
\end{equation}
 is a transitive one. The
  correspondence
  \[
\zeta: \mc{X}_n\rightarrow X_n=\mathrm{Sp}(n,\R)/K, K=
  \mathrm{Sp}(n,\R)\cap\mathrm{O}_{2n}; v\mapsto M_{X+\ii Y}
  K,
\]
where $ M_{X+\ii Y}$ is defined in \eqref{MXY}, is a 1-1 map
  which realizes the Siegel upper half space \eqref{SGUP}  as the
  homogenous manifold 
$X_n$.
\end{Remark}
\begin{proof}
Firstly it is proved that the matrix $cv+d$ in \eqref{conf1} is
invertible, see e.g. \cite[pp 1-11]{pedro}. Then  it is proved that
$M(v)\in\mc{X}_n$ \cite{pedro,sieg}.

It is find a symplectic map that sends $\ii \un$ to $X+\ii
Y\in\mc{X}_n$, $Y>0$ as the  composition of the symplectic maps
$V\rightarrow \sqrt{Y}V\sqrt{Y}$ and  $V\rightarrow V+X$  associated
with the symplectic matrices $$\left(\begin{array}{cc} \sqrt{Y}& \UUN \\    \UUN &\sqrt{Y^{-1}} \end{array}\right) \quad\text{and}\quad
\left(\begin{array}{cc} \un & X \\ \UUN & \un\end{array}\right).$$
We introduce the notation \begin{equation}\label{MXY}M_{X+\ii Y}:=
\left(\begin{array}{cc} \un  & X\\ \UUN  & \un\end{array}\right)
\left(\begin{array}{cc}  \sqrt{Y}& \UUN \\ \UUN & \sqrt{Y^{-1}}
     \end{array}\right) = \left(\begin{array}{cc}\sqrt{Y}
    &X \sqrt{Y^{-1}}\\ \UUN & \sqrt{Y^{-1}} \end{array}\right).\end{equation}
The subgroup of $\mathrm{Sp}(n,\R)$ which   stabilizes $\ii \un\in\mc{X}_n$ is
the subgroup of orthogonal symplectic matrices of the form \eqref{msimor}.
\end{proof}
Note that an argument similar with that used in Remark \ref{direct}
was given in \cite{ali}, following \cite{si}.
\subsection{Pre-Iwasawa and modified   pre-Iwasawa decompositions}
We recall that the  Iwasawa decomposition \cite[Ch VI, \S3]{helg} of  ${\mr{SL}}(2,\R)$ is
used for the  so called $S$-parametrization of the  Jacobi group
$G^J_1(\R)$, see \cite[p 4]{bern84}, \cite[p 15]{BB}, \cite[p 7]{bs}.

In the present paper we find a similar decomposition  for  $G^J_n(\R)$.

We recall the Iwasawa
decomposition \cite[Ch. VI, \S3]{helg} of $\Sp\ni G=KAN$ 
\cite[p. 285]{terras} corresponds to 
$$K:=\left\{\left(\begin{array}{cc} A &B\\ -B &A\end{array}\right)\in
   {\text{ U}}(n)\right\},$$
$$A:=\left\{\text{diag}(a_1,\dots,a_n,a^{-1}_1,\dots,a^{-1}_n);
    a_1,\dots,a_n>0\right\},$$
$$N:=\left\{\left(\begin{array}{cc}A & B \\ \UUN& 
                  (A^{-1})^t\end{array}\right):\text{ A real unit upper
                triangular}, AB^t=BA^t\right\}.$$
          For Cholesky factorisation see \cite[p. 287]{terras} and
          \cite{tam}; 
         for  QR 
         decomposition see \cite[p. 143]{ser}.  We also mention the
         Iwasawa decomposition
         for $G^J_n(\R)$  was considered  in \cite[\S~9.1.2]{yang}. 
         
Following  the method of  \cite{arw} and \cite[\S 2.2.2]{goss}, we
find similarly 
\begin{lemma}\label{R313}{\emph{\bf Pre-Iwasawa decomposition}}
  Let as consider the pre-Iwasawa decomposition of $ M\in \Sp $
  \begin{equation}\label{IWSP} M= \left(\begin{array}{cc} a &b  \\c
                                                    &d\end{array}\right)=\left(\begin{array}{cc}
                                                                                 \un
                                                                                 &x \\
                                                                                 \UUN
                                                                                 &\un\end{array}\right)
                                                                             \left(\begin{array}{cc}
                                                                                     y
                                                                                     &\UUN \\
                                                                                     \UUN
                                                                                     &y^{-1}\end{array}\right)
                                                                                  \left(\begin{array}{cc}
                                                                                     X
                                                                                     &Y 
                                                                                     \\-Y
                                                                                     &X\end{array}\right),
                                                                               \end{equation}
 where the last matrix in \eqref{IWSP}    in  ${\emph{U}}(n)$   verify
 \eqref{msimor}.  

 We
 find  \begin{equation}\label{a0c0}y=(dd^t+cc^t)^{-\frac{1}{2}},~X-\ii
   Y=y(d+\ii c).
 \end{equation}
 Let us also define
 \[
   t:=y^2(db^t+ca^t)y^{-1}=(bd^t+ac^t)y.
   \]
 The matrices $y$ and $x:=ty$ are symmetric,  $y$ is positive definite,  and
 all the factors in \eqref{IWSP} are unique. We have also 
 \begin{equation}\label{Dx}
   x=(dd^t+cc^t)^{-1}(db^t+ca^t)=(bd^t+ac^t)(dd^t+cc^t)^{-1}.
 \end{equation} 
 The inverse of the  transform $(a,b,c,d)\rightarrow (x,y,X,Y)$ in
 equations \eqref{a0c0}, \eqref{Dx}
 is
 \begin{equation}\label{c0a0}
   a= yX-xy^{-1}Y,~b=yY+xy^{-1}X,~c=-y^{-1}Y, ~d=y^{-1}X.\end{equation}
\end{lemma}
  In the case of $\text{SL}(2,\R)$, the expresion \eqref{c0a0} corresponds to
 (46.b) in \cite{SB19} if we replace $y\rightarrow y^{1/2}$ and take
 $X=\cos\theta,~Y=\sin\theta$.
 
The first factor  in \eqref{IWSP} corresponds  to the ``free
propagation subgroup'' \cite{arw}.

In the next Lemma we modify the pre-Iwasawa decomposition of ${\mr{
  Sp}}(n,\R)$ so that it coincides with  Iwasawa decomposition of
the group ${\mr {Sp}}(1,\R)\approx {\mr {SL}}(2,\R)$ in \cite[p
9]{bs}.

We get
 \begin{lemma}\label{L5}{\emph{\bf Modified pre-Iwasawa decomposition}}
           The action   of $M\in\Sp$ 
           \eqref{Real} on  $\mc{X}_n$,
           expressed in the   parameters of the  pre-Iwasawa decomposition in
          {\emph{ Lemma  \ref{R313}}}
          \begin{equation}\label{SCGa}(a,b,c,d)\times (x',y',X',Y')\rightarrow
           (x_1,y_1,X_1,Y_1),\end{equation} where  $x',y'\in M(n,\R),~x'=(x')^t,y'=(y')^t,y'>0$,
         and
          \begin{subequations}\label{c0a01}
           \begin{align}
             a &= y^{1/2}
             X-xy^{-1/2}Y,~b=y^{1/2}Y+xy^{-1/2}X,~c=-y^{-1/2}Y,
                 ~d=y^{-1/2}X,\label{c0a01A}\\
             x &= y(db^t+ca^t),~ y=(dd^t+cc^t)^{-1},~X-\ii
                 Y=y^{\frac{1}{2}}(d+\ii c), 
           \end{align}
         \end{subequations}
           is given by the formulas
           \begin{equation}\label{SPRT}
         \begin{split} x_1+\ii y_1&=[c(y'+x'{y'}^{-1}x')c^t+d(y')^{-1}d^t+cx'(y')^{-1}d^t+d(y')^{-1}x'c^t]^{-1}\\
           &\times[c(y'+x'(y')^{-1}x')a^t+cx'(y')^{-1}b^t+d (y')^{-1}x'a^t+d(y')^{-1}b^t+\ii],
         \end{split}
       \end{equation}
      \begin{equation}\label{X1Y1Z1}
         \begin{split}
         X_1-\ii Y_1 &
         =(y_1)^{1/2}\{(cx'+\!d)(y')^{-1/2}X'+c(y')^{1/2}Y'\\
           &+\ii[c(y')^{1/2}X'-(cx'+d)(y')^{-1/2}Y']\},
       \end{split}
     \end{equation}
     while the action given
           by \eqref{conf1} $M\times (x',y')\rightarrow (\rm{x}_1,\rm{y}_1),
           v_1:=\rm{x}_1+\ii \rm{y}_1$ expressing the linear
           fractional \eqref{conf1}
           transformation is 
          \begin{equation}
           \begin{split}\label{conf2}\rm{x}_1+\ii \rm{y}_1& = (\bar{v'}c^t+d^t)^{-1}(\frac{B}{2}+\ii y')(cv'+d)^{-1},\\
           B &=
               2\bar{v'}a^tcv'+\bar{v'}(c^tb+a^td)+(b^tc+d^ta)v'+2b^td.
             \end{split}
           \end{equation} 
           The modified 
           pre-Iwasawa decomposition \eqref{c0a01} is compatible with
           the M\"obius transform \eqref{conf2},  i. e.
           \begin{equation}\label{X1Y10}
            x_1+\ii y_1\equiv \rm{x}_1+\ii \rm{y}_1.\end{equation}
The transformation of the matrices associated as in {\emph{Remark
\ref{REM2}}}  to the pair   $(X,Y)$  defined in \eqref{X1Y1Z1} under
the action
\eqref{SCGa} reads
\[
  \left(\begin{array}{cc}X_1& Y_1\\-Y_1
                                                                   &
                                                                     X_1 \end{array}\right)=y_1^{\frac{1}{2}}\left[(cx'+d)(y')^{-\frac{1}{2}}\db{
                                                                     1}_{2n}-c(y')^{\frac{1}{2}}J_{2n}\right]\left(\begin{array}{cc}X'&Y'\\-Y'
                                                                   &
                                                                     X'\end{array}\right).
                                                               \]
 \end{lemma}
 \begin{proof}
           This is an easy but long calculation and  we indicate only
           the main steps.

           We write \eqref{SPRT} as
           $$x_1+\ii y_1 = A^{-1}M,$$
           where
           \begin{align*}
           A & :=c(y'+x'{y'}^{-1}x')c^t+d(y')^{-1}d^t+cx'(y')^{-1}d^t+d(y')^{-1}x'c^t,\\
           M & :=c(y'+x'(y')^{-1}x')a^t+cx'(y')^{-1}b^t+d
             (y')^{-1}x'a^t+d(y')^{-1}b^t.
           \end{align*}
           Firstly it is proved that $y_1$ defined in \eqref{SPRT} is
           equal with $\mr{y}_1$ in \eqref{conf2} and  then it is
           obtained 
           \[
             A=(cv'+d)y'^{-1}(\bar{v'}c^t+d^t),
           \]
           or \begin{equation}\label{am1}A^{-1}=(\bar{v'}c^t+d^t)^{-1}y'(cv'+d)^{-1}.\end{equation}
           In order to prove that $x_1$ in \eqref{SPRT} is equal with
           $\mr{x}_1$ in \eqref{conf2}, with \eqref{am1}, we have to
           verify that
           $$y'(cv'+d)^{-1}M=\frac{B}{2}(cv'+d)^{-1},$$
           i.e. \begin{equation}\label{MCV}M(cv'+d)=(cv'+d)(y')^{-1}\frac{B}{2}.\end{equation}
           Using equations \eqref{simplecticR}, it is  verified the
           identity  of the imaginary parts of both sides of
           \eqref{MCV} and then the identity of the real parts.
           \end{proof}
       
\section{The real Jacobi group $G^J_n(\R)$}\label{GHHN}

The real Jacobi group 
 $G^J_n(\R)$ has  the composition law
\begin{equation}\label{GJN}(M,(\lambda,\mu,\kappa)) \times (M',(\lambda',\mu',\kappa'))=
  (MM',(\tilde{\lambda}+\lambda',\tilde{\mu}+\mu',\kappa+\kappa'+\tilde{\lambda}\mu'^t-\tilde{\mu}\lambda'^t),
\end{equation}
where $M,M'\in \Sp$ have the form \eqref{Real} and verifies the conditions
of \ref{simplecticR}, $(\lambda,\mu,\kappa)$,
$(\lambda',\mu',\kappa')\in {\rm H}_n(\R)$, and
$(\tilde{\lambda},\tilde{\mu})=(\lambda,\mu)M'$ \cite{gem,LEE03,SA71C,SA80,yang,ZIEG}.

\subsection{The Jacobi group $G^J_n(\R)$ as subgroup of
  $\text{Sp}(n+1,\R)$} 
Let us consider a matrix $M\in\text{Sp}(n,\R)$ as in \eqref{Real} verifying 
\eqref{simplecticR}. Let us introduce the block matrix 
\begin{equation}\label{GABC}g:= \left(\begin{array}{cc} A & B\\C & D\end{array}\right)\in
M(2n+2,\R), \end{equation}
where  the submatrices $A,B,C,D\in M(n+1,\R)$ are defined as in
\eqref{ggmin1}, i.e. 
\begin{equation}\label{LKJ}
A\!:=\! \left(\!\begin{array}{cc}a_{nn} &
                                           \UU_{n1}\\ \lambda_{1n}
                                         &1_{11}\end{array}\!\right),\!
                                     B\!:=\!\left(\begin{array}{cc}b_{nn} &
                                                                       q^t_{n1}\\ \mu_{1n}
                                           &\kappa_{11}\end{array}\!\right)\!,
C\!:=\! \left(\!\begin{array}{cc}c_{nn} &
                                           \UU_{n1}\\ \UU_{1n} & 0\end{array}\!\right)\!, D\!:=\! \left(\!\begin{array}{cc}d_{nn} &
                                           -p^t_{n1}\\ \UU_{1n}
                                         &1_{11}\end{array}
                                       \!\right).
                                   \end{equation}

 We  verify  that indeed 
 \begin{lemma}\label{L44} The matrix $g$ defined in
   \eqref{GABC}, \eqref{LKJ} is in $\SPn$.
 \end{lemma}
\begin{proof}
We calculate the submatrices of the matrix $L$
$$L:=g J_{n+1}g^t=\left(\begin{array}{cc} U& V \\Z &T\end{array}\right).$$
We find $$U=\UU_{n+1},~V=\db{1}_{n+1},~Z=-\db{1}_{n+1},~T=\UU_{n+1},$$
i.e. $L=J_{n+1}$ and  the conditions  \eqref{XXX} are verified. 
\end{proof}

 If $g=(M,X,\kappa)\in G^J_n(\R)$, then $g^{-1}=(M^{-1},-Y,-\kappa)$,
 i.e., with the conventions in \eqref{GABC}, \eqref{LKJ}, we have the
  following representation in \SPn, see also \eqref{ggmin1}  
  \begin{equation}\label{ggmin}
    g\!=\!\left(\!\begin{array}{cccc}a & 0 & b& q^t\\
              \lambda & 1& \mu &\kappa\\
              c & 0 &  d & -p^t\\
                            0 & 0& 0& 1\!\end{array}\!\right),g^{-1}\!=\!
                        \left(\!\begin{array}{cccc}d^t & 0 & -b^t& -\mu^t\\
              -p & 1& -q &-\kappa\\
              -c^t & 0 &  a^t & \lambda^t\\
              0 & 0& 0& 1\!\end{array}\!\right)\!,\!M\!=\!\left(\!\begin{array}{cccc}a&b\\c&d\end{array}\!\right)\in\!\Sp. \end{equation}
\subsection{The Lie algebra $\got{g}^J_n(\R)$}\label{GOTJ}
Now we introduce a set of matrices that form  a base for the Lie algebra
$\got{g}^J_n(\R)$ embedded in $\got{sp}(n+1,\R)$ as in Lemma \ref{L44} 
which in the case $n=1$ corresponds to the base $F,G,H,P,Q,R$ in
\cite{SB19}
\begin{subequations}\label{alg1}
  \begin{align}
    2(F_{IJ})_{ij} & :=
                     \delta_{I,i}\delta_{J,n+1+j}+\delta_{I,j}\delta_{J,n+1+i},~I,J=1,\dots,2n+2; i,j=1,\dots,n;\\
    2(G_{IJ})_{ij}& :=\delta_{I,n+1+i}\delta_{J,j}+\delta_{I,n+1+j}\delta_{J,i},\\
    (H_{IJ})_{ij}& :=\delta_{I,i}\delta_{J,j}-\delta_{I,n+1+j}\delta_{J,n+1+i},\\
    (P_{IJ})_{ij} & :=
                    \delta_{I,n+1}\delta_{J,j}-\delta_{I,n+1+i}\delta_{J,2n+2},\\
    (Q_{IJ})_{ij} &
                   :=\delta_{I,i}\delta_{J,2n+2}+\delta_{I,n+1}\delta_{J,n+1+i},\\
    R_{IJ}
& :=\delta_{I,n+1}\delta_{J,2n+2} .   
  \end{align}
\end{subequations}
With the conventions \eqref{GABC}, \eqref{LKJ}, the first three
matrices in  \eqref{alg1} can be
written down as
\begin{subequations}\label{alg2}
    \begin{align}
   \label{F1}   2F_{ij}& :=\left(\begin{array}{cccc} 0 & 0 & E_{ij}+E_{ji}& 0\\
                      0& 0& 0& 0\\
                      0& 0& 0& 0\\
                       0& 0& 0& 0\\
                    \end{array}\right),~i,j=1,\dots,n;\\\label{G1}   2G_{ij} & :=\left(\begin{array}{cccc} 0 & 0 & 0& 0\\
                      0& 0& 0& 0\\
                       E_{ij}+E_{ji}& 0& 0& 0\\
                                                           0& 0& 0& 0\\ \end{array}\right),\\
   \label{H1}   H_{ij} & :=  \left(\begin{array}{cccc} E_{ij} & 0 & 0& 0\\
                      0& 0& 0& 0\\
                       0 & 0& -E_{ji}& 0\\
                                                           0& 0& 0&
                                                                    0\\ \end{array}\right),
    \end{align}
  \end{subequations}
  while the matrices $P_p,Q_q$, $p,q=1,\dots, n$, $R$ have been already defined in
  \eqref{algp}.
  
    An element  $X\in \got{g}^J_n(\R)$ can be written as  matrix of
    \SPn~ in the base \eqref{alg2} as 
    \[
      \begin{split}
      X & =\sum_{i,j=1}^na_{ij}H_{ij}+2\sum_{1\le i <j\le n}(b_{ij}F_{ij}+
      c_{ij}G_{ij}) \\ &+\sum_{1\le i =j\le n}(b_{ij}F_{ij}+
      c_{ij}G_{ij})+\sum_{i=1}^n(p_iP_i+q_iQ_i)+rR,~
      b=b^t,~c=c^t. \end{split}
  \]
It can be verified that
    \begin{lemma}The commutation relations of  the generators \eqref{alg2} of the
Jacobi algebra $\got{g}^J_n(\R)$ are
    \begin{subequations}\label{COMM}
      \begin{align}
         [H_{kl},F_{ij}] & =\delta_{lj}F_{ik}+\delta_{li}F_{kj},\\
        [G_{ij},H_{kl}] & =\delta_{ki}G_{lj}+\delta_{kj}G_{li},\\
      4[F_{ij},G_{kl}] &  =
                         \delta_{li}H_{kj}+\delta_{jl}H_{ik}+\delta_{jk}H_{il}+\delta_{ik}H_{jl},\\
        [P_p,Q_q] & = 2\delta_{pq}R,\\
        2[P_p,F_{ij}] & =           \delta_{pi}Q_j+\delta_{pj}Q_i,\\
         2[Q_q,G_{ij}] & = \delta_{iq}P_j+\delta_{jq}P_i,\\
        [P_p,H_{ij}] &  = \delta_{pi}P_j,\\
        [H_{ij},Q_q] & = \delta_{jq}Q_i .
   \end{align}
 \end{subequations}
 \end{lemma}
 The commutation relations \eqref{COMM}  of the generators of
 $G^J_n(\R)$ represent the  generalization of the corresponding commutation relations  (3,4), (5.1) and
 (8.20)  of  the generators  of $G^J_1(\R)$ in \cite{SB19}.

\subsection{The action}
Following \cite[\S 5]{coord}, let us consider the restricted real group
$G^J_n(\R)_0$ consisting of elements of the form defined in
\eqref{GJN}, but $g=(M,X)$, where $X=(\lambda,\mu)$.

We consider the
Siegel-Jacobi upper half space 
 $\mc{X}_n$ realised as in \eqref{SGUP}.

 We introduce  for $\mc{X}^J_n$ the analog of 
parametrization  used in \cite[p 7]{bern84}, \cite[p. 11]{bs}, \cite[\S~38]{cal3} for  $\mc{X}^J_1$
\begin{equation}\label{uvpq}u:=pv+q, ~v:=x+\ii y,~v=v^t,~
  y>0,~p,q\in M(1,n,\R).\end{equation}

It should be noted that there is an isomorphism
$G^J_n(\R)\ni(M,X,K)\rightarrow(M,X)\in  G^J_n(\R)_0$ through which the
action of   $G^J_n(\R)_0$ on  $\mc{X}^J_n$ can be defined as in
\cite[Proposition 2]{coord}.

    It
is easy to prove  that
\begin{lemma}\label{actM}
a)  If $\mc{X}_n\ni v=x+\ii y$, then the action of $G^J_n(\R)_0$ on
$\mc{X}^J_n${\rm :}  $(M,X)\times
  (v',u')\rightarrow (v_1,u_1)$, where $M\in\Sp$  has the expression 
  \eqref{Real}, is given by the formulas
  \begin{subequations}\label{exac}
    \begin{align}
      v_1 & =(av'+b)(cv'+d)^{-1}=(v'c^t+d^t)^{-1}(v'a^t+b^t),\label{exac1}\\
      u_1 & = (u'+\lambda v'+\mu)(cv'+d)^{-1}. \label{exac2}
    \end{align}
  \end{subequations}
 If the modified pre-Iwasawa  decomposition \eqref{c0a01} is used, $v_1$ in
 \eqref{exac1} has the equivalent expressions \eqref{SPRT},
 \eqref{conf2} via the identification \eqref{X1Y10}.

  b)  For $\lambda, \mu\in M(1,n,\R)$, let us consider $(p,q)$ such that 
  \begin{subequations}\label{pqlm}
    \begin{align}
  (p,q) & =(\lambda,\mu)M^{-1}=(\lambda d^t-\mu c^t, -\lambda b^t+ \mu a^t),\\
      (\lambda,\mu) & =(p,q)M=(pa+qc,pb+qd),
~~~ p,q,\lambda,\mu\in M(1,n,\R).
\end{align}
  \end{subequations} 
 Then  the action of $G^J_n(\R)_0$ on $\mc{X}^J_n${\rm :}   $(M,X)\times
  (x',y',p',q')\rightarrow (x_1,y_1,p_1,q_1)$ is given by
  \eqref{exac1}, while
\begin{equation}\label{p1q1}
(p_1,q_1)=(p,q)+(p',q')\left(\begin{array}{cc}a & b \\c &
                                                          d\end{array}\right)^{-1}=
                                                      (p+p'd^t-q'c^t,q-p'b^t+q'a^t).
                                                       \end{equation}
 c) The action of $G^J_n(\R)$ on $\tilde{\mc{X}}^J_n\approx
 \mc{X}^J_n\times \R${\rm :}  
\begin{subequations}\label{actc}
\begin{align}
 (M,(\lambda,\mu),\kappa)\times (v',u',\kappa') &\rightarrow (v_1,u_1,\kappa_1),\\
(M,(\lambda,\mu),\kappa)\times
  (x',y',p',q',\kappa')& \rightarrow (x_1,y_1,p_1,q_1,\kappa_1)
\end{align}
\end{subequations}
 is given by
  \eqref{exac}, \eqref{p1q1} and
  \[
    \kappa_1=\kappa+\kappa'+\lambda q'^t-\mu p'^t.
    \]
 d)  The $1$-form
  \begin{equation}\label{1FORMR}
    \lambda^R=\dd \kappa -p \dd q^t+q \dd p^t
  \end{equation}
  is invariant to the action \eqref{actc} of $G^J_n(\R)$ on
  $\tilde{\mc{X}}^J_n$.

  e) The action of $G^J_n(\R)$ on $G^J_n(\R)$ $$(M,(\lambda,\mu),\kappa)\times
  (S_n)'\rightarrow
  (S_n)_1,$$ is
  given in \eqref{X1Y1Z1} for $X',Y'$, while the other actions are
  given in {\emph{ a)-d)}}
  of the present Lemma.
\end{lemma}
  
  \subsection{Fundamental vector fields on $\mc{X}^J_n$ and $\tilde{\mc{X}}^J_n$}\label{FVF}

  We calculate FVF 
  associated to  the 
  generators of the Jacobi group  on  homogenous manifolds
  attached  to $G^J_n(\R)$.

            For a symmetric matrix $x\in M(n)$ we introduce the
            notation
            \begin{equation}\label{px}
              \pa_x:=\left((2-\delta_{ij})\frac{\pa }{\pa
                  x_{ij}}\right)_{i,j=1,\dots,n}.
            \end{equation}
            If $a=(a_1,\dots,a_n)$, $b=(b_1,\dots,b_n)$ are two
            $n$-vectors,  we introduce also the
            notation
            \begin{equation}\label{py}
            (a\odot
            b)_{ij}:=a_ib_j+a_jb_i-a_ib_j\delta_{ij},\quad ij=1,\dots,n.
          \end{equation}

          Note the isomorphism of the representations  \eqref{HFV} 
          and  \eqref{alg2}  of $\got{sp}(n,\R)$.  To a matrix $A$ as
          in \eqref{alg2}, let us denote by $\hat{A}$
          the corresponding matrix in the representation \eqref{HFV}.

          We make the following 
\begin{Remark}\label{REM10}
                Let $z\in M(n)$. Then we have
                the relation  \eqref{EQN}
                \begin{equation}\label{EQN}
         \frac{\pa}{\pa
          z}\dd z=\un, \text{ ~i.e.~} \frac{\pa z_{ij}}{\pa
          z_{pq}}=\delta_{ip}\delta_{jq} .
      \end{equation}
      If the matrix  $z$ is symmetric,  instead of \eqref{EQN} we have \eqref{R1}
      \begin{equation}\label{R1}
           D_z\dd z= \un,~z=z^t, \text{ ~i.e.~}
            (D_z)_{\mu\nu}\dd
            z_{\nu\chi}=\delta_{\mu\chi},~z_{\mu\nu}=z_{\nu\mu},
          \end{equation}
          where    
            \begin{equation}\label{DZ}
              (D_z)_{\mu\nu}:=e_{\mu\nu}\frac{\pa}{\pa
            z_{\mu\nu}},~e_{\mu\nu}:=\frac{1+\delta_{\mu\nu}}{2},
          \text{~~no summation!}
          \end{equation}
              \end{Remark}
              \begin{proof}
                \eqref{EQN} is evident.
                
Using equation \cite[(4.5)]{coord} which says that for a symmetric matrix $w$ we have
\begin{equation}\label{pawpaw}\frac{\pa
                w_{ij}}{\pa w_{pq}}=\delta_{ip}\delta_{jq}+\delta_{iq}\delta_{jp}-\delta_{ij}\delta_{pq}\delta_{ip},
              ~w_{ij}=w_{ji},\end{equation}
\eqref{R1} it is verified, where the symbol  $D$ in \eqref{DZ}
was introdced in \cite[(3.39)]{coord}.
  \end{proof}

We obtain the following representations of the FVF
associated to the base \eqref{alg2}, \eqref{algp} of  the Lie algebra
 $\got{g}^J_n(\R)$ 
          \begin{Proposition}\label{FDVF}
            a) The fundamental vector fields in the coordinates
   $(v,u)$ of $\mc{X}^J_n$ on which $G^J_n(\R)$ acts as in {\emph{Lemma
   \ref{actM} a)}} are given by the holomorphic FVF
   \begin{subequations}\label{EQQ1}
     \begin{align}
   F^*_{ij} &=\hat{F}_{ij}\frac{\pa}{\pa v},\quad i,j=1,\dots,n;\label{411a}\\
   {G}^*_{ij} & =-v\hat{G}_{ij}v\frac{\pa}{\pa
     v}-(\frac{\pa}{\pa
     u})^tu\hat{G}_{ij}v;\label{411b}\\H^*_{ij} &=(\hat{E}_{ij}v+v\hat{E}_{ji})\frac{\pa}{\pa
    v}+(\frac{\pa}{\pa u})^t u\hat{E}_{ji};\label{411c}\\
   P^*_p& =\hat{E}_pv(\frac{\pa }{\pa
          u})^t;~Q^*_q=\hat{E}_q(\frac{\pa}{\pa u})^t;~R^*=0,\quad p,q=1,\dots,n.\label{411d}
  \end{align}\end{subequations}
    b) The real holomorphic FVF associated  to \eqref{EQQ1}  in the
    variables $(x,y,\xi,\rho)$ on $\mc{X}^J_n$, where $v:=x+\ii y$,
    $y>0$, $u:=\xi+\ii \rho$ as in \eqref{uvpq}, are
    \begin{subequations}\label{EQQ2}
      \begin{align}
        F^*_{ij}  & =(F^*_1)_{ij}, \label{412a}\\
      G^*_{ij} & = (G^*_1)_{ij} +(\frac{\pa}{\pa
                  \xi})^t (\rho
                  \hat{G}_{ij}y-\xi\hat{G}_{ij}x)-(\frac{\pa}{\pa
                 \rho})^t (\xi\hat{G}_{ij}y+\rho\hat{G}_{ij}x) ,\label{412b}\\
        H^*_{ij}   & =  (H^*_1)_{ij}+(\frac{\pa}
                 {\pa  \xi})^t\xi\hat{E}_{ji}+(\frac{\pa}{\pa
                    \rho})^t\rho\hat{E}_{ij},\label{412c}\\
      P^*_p& =\hat{E}_px(\frac{\pa}{\pa \xi})^t+\hat{E}_py(\frac{\pa}{\pa
               \rho})^t;~Q^*_q=\hat{E}_q(\frac{\pa}{\pa
             \xi})^t,~R^*=0, \label{PQRpqr}
             \end{align}
           \end{subequations}
           where 
\begin{subequations}\label{FDFGH}
      \begin{align}
           (F^*_1)_{ij}& =\hat{F}_{ij}\frac{\pa }{\pa x},~\\
        (G^*_1)_{ij}  & =\alpha \frac{\pa}{\pa
                  x}-\beta\frac{\pa }{\pa
                  y},\\ 
       (H^*_1)_{ij}& =  (\hat{E}_{ij}x+x\hat{E}_{ji})\frac{\pa}{\pa
        x}+(\hat{E}_{ij}y+y\hat{E}_{ji})\frac{\pa}{\pa
                     y},\\
    \alpha &:=    (y\hat{G}_{ij}y-x\hat{G}_{ij}x), ~\beta:=(x\hat{G}_{ij}y+y\hat{G}_{ij}x),
                      \end{align}
                    \end{subequations}
are FVF associated with
the generators of $\got{sp}(n,\R)$ corresponding to
the action \eqref{exac1} of \Sp~ on $\mc{X}_n$.
                    
 c) The FVF  \eqref{EQQ2}  in  the
 variables $(x,y,p,q)$  on $\mc{X}^J_n$, where
 \begin{subequations}\label{schpq}
 \begin{align}v&=x+\ii y,~u=pv+q=\xi+\ii \rho,\\
               & p=\rho y^{-1}, q=\xi-\rho y^{-1}x,
 \end{align}
  \end{subequations} are
\begin{subequations}\label{EQQ3}
  \begin{align}
   ( F^*)_{ij}\!= &\hat{F}_{ij}(\frac{\pa}{\pa _x}-\frac{\pa}{\pa
                 q}\odot p), \label{415a} \\
    (G^*)_{ij} = &  (G^*_1)_{ij} -((\frac{\pa}{\pa
                  p})^tp)y^{-1}\beta +\beta(\frac{\pa}{\pa
                   p}y^{-1})\odot p 
  -\alpha\frac{\pa}{\pa q}\odot p 
               +((\frac{\pa}{\pa
    q})^tp)\alpha\nonumber\label{415b}\\
                     & \!-\!\beta(\frac{\pa}{\pa
    q}xy^{-1})\odot p   \!+\! ((\frac{\pa}{\pa
                       q})^tp)(y^{-1}x\beta) \\
                   & - ((\frac{\pa}{\pa
                     p})^tq) y^{-1}\hat{G}_{ij}y +((\frac{\pa}{\pa q})^tq)y^{-1}\alpha, \nonumber\\
   ( H^*)_{ij}= & (H^*_1)_{ij} +(\hat{E}_{ij}y+y\hat{E}_{ji})[-(\frac{\pa}{\pa
                    p}y^{-1})\odot p + (\frac{\pa}{\pa q}xy^{-1})\odot
                    p] \nonumber\\ & -(\hat{E}_{ij}x+x\hat{E}_{ji})\frac{\pa }{\pa
                    q}\odot p
     +((\frac{\pa}{\pa q})^tp)(x\hat{E}_{ji}-y^{-1}xy\hat{E}_{ji}) \label{415c}\\ &+((\frac{\pa}{\pa
    q})^tq)\hat{E}_{ji}+((\frac{\pa}{\pa p})^tp)\hat{E}_{ij}, \nonumber\\
    P^*_p \!=  & E_p(\frac{\pa}{\pa p})^t,~ Q^*_q= E_q(\frac{\pa}{\pa q})^t,~R^*=0.
\end{align}
\end{subequations}

d) Now we consider the action of  $G^J_n(\R)$ on $(u',v',\kappa')\in\tilde{X}^J_n$ as in {\emph{Lemma
\ref{actM} c)}}. We find for FVF
$F^*_{ij},G^*_{ij},H^*_{ij}$ the expressions \eqref{411a},
\eqref{411b}, respectively \eqref{411c}, while instead of \eqref{411d}, we find
\[
   P^*_p =\hat{E}_pv(\frac{\pa }{\pa u})^t+q\pa_{
   \kappa};~Q^*_q=\hat{E}_q(\frac{\pa}{\pa u})^t-p\pa_{\kappa}
 ;~R^*=\pa_{\kappa}.
 \]

 e)  The FVF on $\tilde{\mc{X}}^J_n$ in the variables
 $(x,y,\xi,\rho,\kappa)$ are given by \eqref{412a}, \eqref{412b},
 respectively \eqref{412c}
 for $F^*_{ij}, G^*_{ij}, H^*_{ij}$, while \eqref{PQRpqr} became
 \[
   P^*_p=E_p(x\frac{\pa}{\pa \xi}+y\frac{\pa}{\pa
   \rho})+q\pa_{\kappa}, Q^*_q=E_q\frac{\pa}{\pa
   \xi}-p\pa_{\kappa},R^*=\pa_{\kappa}, ~p=\rho y^{-1},~q=\xi-\rho
 y^{-1}x.
\]

 f) We express the FVF  $F^*_{ij},G^*_{ij},H^*_{ij}$  on $\tilde{\mc{X}}^J_n$ in
 the variables  $(x,y,p,$ $q,\kappa)$ as in
 \eqref{415a},  \eqref{415b}, respectively \eqref{415c}, and
  \begin{equation}\label{EQQ6}
 P^*_p=E_p(x\frac{\pa}{\pa q}-y\frac{\pa}{\pa q}xy^{-1}+y\frac{\pa}{\pa
  p}y^{-1})+q\pa_{\kappa},~ Q^*_q=E_q\frac{\pa}{\pa
  q}-p\pa_{\kappa},~R^*=\pa_{\kappa}.
\end{equation}

  \end{Proposition}
  \begin{proof}

            a) We apply the definition of fundamental vector fields.
            For $P_p,Q_q,R$  on components, we find $$(P^*_p)_i  = (\hat{E}_pv)_i\frac{\pa }{\pa
                    u_i},~(Q^*_q)_i=(\hat{E}_q)_i\frac{\pa}{\pa
                    u_i},~~~R^*=0,$$
                  which we write as  in \eqref{411d}.
            
            b)   In order to determine the real holomorphic FVF
            associated to the holomorphic FVF \eqref{EQQ1},   let  $Z$ be  a
            holomorphic 
            vector field  on a complex  $n$-dimensional   manifold
            $$ Z:=\sum_{i=1}^nZ_j\frac{\pa}{\pa z_j}, ~Z_j:=A_j+\ii
            B_j,~ A_j,B_j\in C^{\infty}(M). $$
            Then the real holomorphic field $X=Z+\bar{Z}$ in
            coordinates $(x_j,y_j)$, $z_j=x_j+\ii y_j$ is, see
            \cite[Proposition 22 in 
            v1]{SB19} or
            \cite[Proposition 2.11]{kn}, 
            $$X=\sum_{i=1}^nA_j\frac{\pa}{\pa x_j}+B_j\frac{\pa}{\pa y_j}.$$
            
            c) In order to make the change of variables
            $(x,y,\xi,\rho)\rightarrow (x,y,p,q)$ as in \eqref{schpq},
           firstly it is observed that the Jacobian of the
            transformation  is non-zero:
            $\frac{\pa(x,y,\xi,\rho)}{\pa(x,y,p,q)}=-y<0.$

            With formula \eqref{pawpaw}, we get the following formulas
  \begin{subequations}
    \begin{align*}
      \frac{\pa}{\pa x_{ij}} & \rightarrow
                               (2-\delta_{ij})\frac{\pa}{\pa
                               x_{ij}}-(p\odot \frac{\pa}{\pa
                               q})_{ij};\\\frac{\pa}{\pa
                               y_{ij}} &\rightarrow
                               (2-\delta_{ij})\frac{\pa}{\pa
                               y_{ij}}-(\frac{\pa}{\pa p}y^{-1}\odot
                               p)_{ij}  +(\frac{\pa}{\pa q} xy^{-1}\odot
                               p)_{ij}; \\\frac{\pa}{\pa
                               \xi_i} & \rightarrow \frac{\pa}{\pa
                               q_i};\\\frac{\pa}{\pa
                               \rho_l} & \rightarrow (\frac{\pa}{\pa
                              p}y^{-1})_l-(\frac{\pa}{\pa q}xy^{-1})_l;
\end{align*}
\end{subequations}
which can be written down  in  the conventions  \eqref{px}, \eqref{py} as 
    \begin{subequations}
              \begin{align*}
                \frac{\pa}{\pa x} &=\pa_x-\frac{\pa}{\pa q}\odot p;\\
                \frac{\pa }{\pa y}
                                       &=\pa_y+[(-\frac{\pa}{\pa
                                         p}+\frac{\pa }{\pa
                                         q}x)y^{-1}]\odot p; \\
                \frac{\pa }{\pa \xi} & = \frac{\pa}{\pa q};\\
                \frac{\pa}{\pa\rho} &= \frac{\pa}{\pa
                                      p}y^{-1}-\frac{\pa}{\pa q}xy^{-1}.
               \end{align*}
              \end{subequations}
            \end{proof}

\subsection{Invariant one-forms on the Jacobi group}\label{IOF}

        From \eqref{ggmin},  we obtain 
        \begin{equation}\label{gmdg}
          g^{-1}\dd g = \left(\begin{array}{cccc} A_{11} & A_{12} &
                                                                    A_{13}&
                                                                            A_{14}\\
                                A_{21} & A_{22}& A_{23} & A_{24}\\
                                A_{31} & A_{32} & A_{33} & A_{34}\\
                                A_{41}& A_{42} & A_{43} & A_{44}
                              \end{array}\right),
                          \end{equation}
  where
  \begin{equation}\label{A11A22}
    \begin{split}
        A_{11}& \!=\!d^t\dd a -b^t\dd c;A_{12}\!=\!0;A_{13}\!=\!d^t\dd b-b^t \dd
              d; A_{14}\!=\!d^t\dd q^t+b^t\dd p^t;\\
      A_{21} &\!=\!\dd\lambda\!-p\! \dd a \!-\! q\dd c;A_{22}\!=\! 0; A_{23}\!=\!\dd
               \mu\!-\!p\dd b\!-\!q\dd d;A_{24}\!=\! \dd \kappa -p\dd q^t\!+\!q\dd p^t;\\
      A_{31}& \!=\! -c^t\dd a +a^t \dd c; A_{32}\!=\! 0;A_{33}\!=\!-c^t\dd b+a^t
              \dd d;A_{34}=-c^t\dd q^t-a^t \dd p^t;\\
      A_{41}&\!=\!A_{42}=A_{43}=A_{44}=0.
       \end{split}
     \end{equation}
     With \eqref{pqlm} and \eqref{simplecticR}, we get from
     \eqref{A11A22} the relations
     \begin{equation}\label{REL}
       A_{24}=\dd \kappa-p \dd q^t+q\dd p^t;~
         A_{34}=-A_{21}^t;~A_{23}=A^t_{14}; ~A_{11}=-A_{33}^t.
       \end{equation}
       With \eqref{gmdg} and \eqref{REL}, we get
       \begin{lemma}\label{L9}
         For $g\in\got{g}^J_n(\R)$ as in \eqref{ggmin}, we have in the
         basis \eqref{alg2}, \eqref{algp} the expression
         \[
           g^{-1}\!\dd\! g
         \!=\!\sum_{i,j=1}^n\!(\lambda^H)_{ij}H_{ij}\!+\!\sum_{1\le i\le j\le n
           }[(\lambda^F)_{ij}F_{ij}\!+\!
           (\lambda^G)_{ij}G_{ij}]\!+\!\sum_{i=1}^n[(\lambda^P)_iP_i\!+\!(\lambda^Q)_iQ_i]\!+\!\lambda^R
           R,
           \]
         where the invariant one-forms corresponding to the generators
         \eqref{alg2} are
         \begin{subequations}\label{CEDE}
           \begin{align}
             \lambda^F &=d^t\dd b-b^t \dd
              d =(\lambda^F)^t,\\
             \lambda^G& = -c^t\dd a +a^t \dd c=(\lambda^G)^t, \\
             \lambda^H&=d^t\dd a -b^t\dd c=\dd b^t c-\dd d^t a =(\lambda^H)^t,\\
             \lambda^P & =\dd \lambda -p 
                         \dd a- q\dd c = \dd p a+\dd
                         qc=\lambda^p-\lambda \lambda^H-\mu\lambda^G,\\
             \lambda^Q&= \dd q d+\dd p b=\dd \mu-p \dd
                        b- q\dd d=\lambda^q-\lambda\lambda^F+\mu\lambda^H,\\
             \lambda^R&= \dd \kappa -p\dd q^t+q\dd p^t=\lambda^r+\lambda\lambda^F\lambda^t-\mu\lambda^G\mu^t-2\lambda\lambda^H\mu^t,
             \end{align}
           \end{subequations}
           and $\lambda^p,\lambda^q,\lambda^r$ are given by
           \eqref{lplqlr}.

  Let us introduce the notation
  \[
                L:=y^{-1}\dd y,\quad R:=\dd y y^{-1},\quad C:=y^{-1}\dd
                x y^{-1}.
              \]
         
         With {\emph{Lemma  \ref{R313}}}, we rewrite the invariant one-forms
         \eqref{CEDE} 
          for $G^J_n(\R)$ as
         \begin{subequations}\label{LFGHL}
           \begin{align}
            \label{FF1} \lambda^F & = X^t\dd Y-Y^t\dd X+X^tL Y+
                         X^tCX+Y^tRX,\\
           \label{GG2}  \lambda^G & =-X^t\dd Y+ Y^t\dd X+ Y^tL X
                         -Y^tCY + X^t  RY,\\
          \label{HH2}   \lambda^H& = X^t\dd X+Y^t\dd Y +X^tLX-X^tC Y- Y^tRY,\\
             \label{PP2}           \lambda^P & = \dd p (yX-xy^{-1}Y)- \dd
             qy^{-1}Y,\\
         \label{QQ2}    \lambda^Q & =\dd q y^{-1}X +\dd p(yY+xy^{-1}X),\\
           \label{RR2}  \lambda^R & = \dd \kappa -\dd q p^t+\dd p q^t.
           \end{align}
         \end{subequations}
         We have also
         \begin{equation}
           \begin{split}\label{313}
             \lambda^F+\lambda^G & \!=\!X^t(L+R)Y\!+\!Y^t(L+R)X\!
    +\!X^tCX\!-\!Y^tCY,\\
             \lambda^F\!-\!\lambda^G & \!=\!2(X^t\dd Y\!-\!Y^t\dd X)\!+\!2X^t(L-R)Y+\!X^tCX\!+\!Y^tCY.
             \end{split}
           \end{equation}
     \end{lemma}
         Equations \eqref{CEDE} generalize to $G^J_n(\R)$, $n\in\N $,
         the corresponding equations
         (4.4) and (5.19) in \cite{SB19} for $G^J_1(\R)$. The last
         expression of $\lambda^R$ was obtained previously in
         \eqref{1FORMR} just in analogy to  \cite[(5.5f)]{SB19} for
         the Jacobi group $G^J_1(\R)$ and the
         invariance of the 1-form was verified.
 
           We see in \eqref{313} that  for any $n\in \{\N\}\setminus \{1\}$, 
           $\lambda^F+\lambda^G$ does not depend on $\dd X,\dd Y$,
           but  $\lambda^H$ does, while in the case $n=1$ both $\lambda^F+\lambda^G$
           and $\lambda^H$ they does not depend on $\dd\theta$.

           Indeed, 
          in the case of $G^J_1(\R)$, $X=\cos\theta$, 
         $Y=\sin\theta$, $y\rightarrow y^{\frac{1}{2}}$,  we get
         equations (4.11) in \cite{SB19}
         \begin{subequations}
           \begin{align}\label{capa}
             \lambda^F & =\frac{\dd x}{y}\cos^2\theta +\frac{\dd
                         y}{2y}\sin 2\theta +\dd \theta,\\
             \lambda^G& =-\frac{\dd x}{y}\sin^2\theta+ \frac{\dd
                        y}{2y}\sin 2\theta-\dd\theta,\\
             \lambda^H & = -\frac{\dd x}{2y}\sin 2\theta +\frac{\dd
                         y}{2y}\cos 2\theta ,\\
             \lambda^F+\lambda^G & = \frac{\dd x}{y}\cos 2\theta
                                   +\frac{\dd y}{y}\sin 2\theta,\label{capapl}\\
             \lambda^F-\lambda^G & = \frac{\dd x}{y}+2\dd \theta,\\
             2\lambda^H &= -\frac{\dd x}{y}\sin 2\theta +\frac{\dd
             y}{y}\cos 2\theta. 
              \end{align}
            \end{subequations}
          
            With the first equation \eqref{313}, we get
            \begin{equation}
              \begin{split}\label{LFLG2}
                (\lambda^F\!+\!\lambda^G)^2 &
               \! =\!\tr[2(X^t\!LY X^t\!RY\!\!+\!\!X^tLYX^t\!RY\!\!+\!\!X^tLYY^t\!LX\!\!-\!\!X^tLYY^tCY)\\
  &{\mbox{~~~~}}+X^tLYY^tRX+X^tLYX^tCX\\
  &{\mbox{~~~~}}+2(X^tRYX^tLY+X^tRYX^tRY-X^tRYY^tCY)\\
  &{\mbox{~~~~}} + X^tRYY^tLX+X^tRYX^tCX \\
  &{\mbox{~~~~}} +2(Y^tLXX^tLY+Y^tLXX^tCX-Y^tLXY^tCY)+Y^tLXX^tLY\\
  &{\mbox{~~~~}}+2(Y^tRXX^tCX-Y^tRXY^tCY)\\
  &{\mbox{~~~~}}+Y^tRXX^tLY+X^tCXY^tRX+X^tCXY^tL X].
  \end{split}
\end{equation}
With \eqref{HH2}, we get
\begin{equation}\label{LH2}
  \begin{split}
    (\lambda^H)^2 &=  (\lambda^H_1)^2+(\lambda^H_2)^2,\\
    (\lambda^H_1)^2 & =\tr[ X^t(LX\!-\!CY)X^t\dd X\!+\!X^t\dd X X^t(\dd X\!+\!LX)\!+\!
    Y^t(\dd Y\!-\!RY)Y^t\dd Y\\
    &{\mbox{~~~~}}-Y^t\dd Y(Y^tR+X^tC)Y+X^t\dd X Y^t(\dd Y-RY)-X^t\dd
    X X^tCY\\
   & {\mbox{~~~~}}+X^t(LX-CY)Y^t\dd Y +Y^t\dd Y X^t(\dd
   X+LX)-Y^tRYX^t\dd X] ;\\
   (\lambda^H_2)^2 &= \tr\{X^t(LX-CY)X^tLX+Y^tRY(Y^tR+X^tC)Y\\
   & {\mbox{~~~~+}} X^t[(CY-LX)Y^tR-RXY^tL+(CY-LX)X^tC]Y\}.
  \end{split}
\end{equation}
With the second equation \eqref{313}, we get 
\begin{equation}\label{FMING}
   \begin{split}
  (\lambda^F-\lambda^G)^2 &= \tr\{4[(X^t\dd Y-Y^t\dd X)^2+(X^t\dd
  Y-Y^t\dd X)(X^tCX+Y^tCY)]\\
  & {\mbox{~~~~}}+2XCX^tYCY^t+(X^tCX)^2+ (Y^tCY)^2+[X^t(L-R)Y]^2\\
   & {\mbox{~~~~}}2(X^tCX\!\!+\!\!Y^tCY)[X^t(L\!\!-\!\!R)Y\!\!+\!\!Y^t(R\!\!-\!\!L)X]\!\!\\
   &{\mbox{~~~~}}+ 4X^t(L\!\!-\!\!R)(X^t\dd Y\!\!-\!\!Y^t\dd X)\}.
  \end{split}
\end{equation}
In the case of the Jacobi group $G^J_1(\R)$, when
$X=\cos\theta,Y=\sin\theta$ and $y\rightarrow y^{1/2}$,
\eqref{LFLG2}, \eqref{FMING}, respectively \eqref{LH2} become what is obtained from
\eqref{capapl}, i.e.
\[
  \begin{split}
(\lambda^F+\lambda^G)^2 &= \frac{(\cos 2\theta \dd x)^2+(\sin 2\theta
  \dd y)^2 +\sin4\theta \dd x\dd y}{y^2};\\
(\lambda^H_1)^2 &= 0;\\
(\lambda^H_2)^2& =\frac{(\sin 2\theta  \dd x)^2+(\cos
  2\theta \dd y)^2-\sin 4\theta\dd x\dd y}{4y^2};\\
(\lambda^F-\lambda^G)^2& = 4\dd \theta^2+4\frac{\dd x \dd
  \theta}{y}+\frac{\dd x^2}{y^2}.
\end{split}
\]

            \subsection{Invariant vector fields on the Jacobi group}\label{difi}
            Once we have determined the invariant one-forms
            \eqref{CEDE}, we have to  determine the invariant vector fields
            orthogonal to them solving the equations
            \begin{equation}\label{lLinv}
              <\lambda^{\alpha}|(L^{\beta})^t>=\delta_{\alpha\beta},\alpha,\beta
            =
            F,G,H;~<(L^{\alpha})^t|\lambda^{\beta}>=\delta_{\alpha\beta}, \alpha,\beta=P,Q,R. \end{equation}
            We find
            \begin{subequations}\label{paLLL}
              \begin{align}
                (L^F)^t& =(\frac{\pa}{\pa_b})a+(\frac{\pa}{\pa_d})c,\\
                (L^G)^t& =(\frac{\pa}{\pa_a}+\frac{\pa}{\pa_b})b+(\frac{\pa}{\pa_c}+\frac{\pa}{\pa_d})d,\\
                (L^H)^t& =(\frac{\pa}{\pa_a})a+(\frac{\pa}{\pa_b})b+(\frac{\pa}{\pa_c})c+(\frac{\pa}{\pa_d})d,\\
                L^P &=(\frac{\pa}{\pa_p})d-(\frac{\pa}{\pa_q})b-(\frac{\pa}{\pa_{\kappa}})(pb+qd),\\
                L^Q &= -(\frac{\pa}{\pa_p})c+(\frac{\pa}{\pa_q})a+(pa+qc)\frac{\pa}{\pa_{\kappa}},\\
                L^R& =\frac{\pa}{\pa_{\kappa}}.
                \end{align}
              \end{subequations}
              
    In order to determine the invariant vector fields orthogonal
    to the invariant one-forms \eqref{LFGHL} as in \eqref{lLinv}, we
    have to calculate the derivative of $(a,b,c,d)$ expressed as in
    pre-Iwasawa decomposition \eqref{a0c0}, \eqref{Dx} or modified   pre-Iwasawa
    decomposition
    \eqref{c0a01}, but this is not an easy
    task.

    For exemple, let's take   the simpler case  $d =y^{-1}X$ in \eqref{a0c0}, then $\dd d=-y^{-1}\dd y
    y^{-1}+y^{-1}\dd X.$ More generally, let us   consider the one-forms 
    $$F_{pq}:=A_{pi}\dd y_{ij}B_{jq}+C_{pi}\dd X_{ij} D_{jq}, \quad
    A,B,C,D \in M(n,\R).$$
    We have to determine the invariant vector
    field $$f_{qr}:=M_{qm}D_{mn}(y)N_{nr}+P_{qm}D_{mn}(X)Q_{nr}$$
  such that $$<F_{pq}|f_{qr}>=\delta_{pr}.$$
    The matrices $M,N,P,Q$ such that satisfy the following 
    matrix equation
    \[
      \tr[(MB)(AN)+(CP^tQ)]=\un,
      \]
must be determined, 
    which is generally a difficult problem. If we consider the
     expression of $d$ in \eqref{c0a01A} the situation is even more
     complicated because of the difficulties to calculate the
     differential of the square root  of a matrix, see Appendix
     \ref{AP8}, and  we abandon the task  of explicitly determining 
    the invariant vector fields orthogonal  to the left
    invariant one-forms \eqref{CEDE}.
            
        \section{Invariant metrics on   homogeneous
          manifolds associated  to  $G^J_n(\R)$}
We follow the  notation in \cite[(4.15), (5.21)]{SB19}  for the invariant one-forms on
$G^J_1(\R)$.
\begin{Proposition}\label{PRm}Let  us introduce the invariant one-forms
on $G^J_n(\R)$
 \begin{equation}\label{BBB}
\begin{split}
            \lambda_1& :=\sqrt{\alpha}(\lambda^F+\lambda^G),~ \lambda_2
            :=\sqrt{\alpha}\lambda^H,~\lambda_3:=\sqrt{\beta}(\lambda^F-\lambda^G),\\
            \lambda_4&
            :=\sqrt{\gamma}\lambda^P,~\lambda_5:=\sqrt{\gamma}\lambda^Q,~\lambda_6:=\sqrt{\delta}\lambda^R,\quad \alpha,\beta,\gamma,\delta>0,
          \end{split}
          \end{equation}
         where we use the expressions \eqref{LFGHL} for
         $\lambda^F,\dots,\lambda^R$. The composition law \eqref{GJN}  in the variables
         $(x,y,X,Y)$ is given in {\emph{Lemma \ref{R313}}} or in {\emph{Lemma \ref{L5}}}, and  for
         $p,q,\kappa $ in {\emph{Lemma \ref{actM}}}.
         Let us consider the {\emph{4}}-parameter  left invariant metric
         on $G^J_n(\R)$, which coincides with metric {\emph{(5.32)}} on $G^J_1(\R)$ in \cite{SB19}
         \begin{equation}\label{MMARE}
          \dd s^2_{G^J_n(\R)}=
         \sum_{i=1}^6\lambda_i^2,
       \end{equation}
           where the square of the invariant one-forms $\lambda_1,\lambda_2,\lambda_3$ in
           \eqref{MMARE} are given in \eqref{LFLG2}, \eqref{LH2},
           respectively 
           \eqref{FMING},  and the squares of
           $\lambda_4,\lambda_5,\lambda_6$ are given taking the square
           of\eqref{PP2} \dots \eqref{QQ2}.
           
           Depending of the values of the parametres
           $\alpha,\beta,\gamma,\delta$, \eqref{MMARE} gives the  invariant metric on
the following manifolds{\rm :} 
\begin{enumerate}
\item{\text{~~~~~~~~}}if $\beta,\gamma,\delta =0$ - the  Siegel upper half-plane $\mc{X}_n${\rm{;}}
\item{\text{~~~~~~~~}}if $\gamma,\delta=0, \alpha\beta\not= 0$ - the group $\Sp${\rm{;}} 
\item{\text{~~~~~~~~}}if $\beta, \delta= 0$ - the Siegel-Jacobi half  space
$\mc{X}^J_n${\rm{;}} 
\item{\text{~~~~~~~~}}if $\beta=0$ - the extended
  Siegel-Jacobi extended  half space
$\tilde{\mc{X}}^J_n${\rm{;}}  
\item{\text{~~~~~~~~}}if $\alpha\beta\gamma\delta\not= 0$ - the Jacobi
  group  $G^J_n(\R)$.
  \end{enumerate}

  The invariant vector fields \eqref{paLLL}, orthonormal with respect
  the invariant one-forms \eqref{BBB} in the sense of \eqref{lLinv},  are orthonormal with respect
  to the metric \eqref{MMARE}.
\end{Proposition}
Proposition \ref{PRm} is an extension to $G^J_n(\R)$, $n\in \N$, of  \cite[Theorem 1]{SB19} for
$G^J_1(\R)$.   However, the expressions \eqref{LFLG2}, \eqref{LH2},
           \eqref{FMING} are  complicated and also the invariant
           vector fields \eqref{paLLL} are in fact not explicitly calculated due to
           the difficulties signaled in Section \ref{difi}. Even the metric
           on the Siegel upper-half space given at (1) in Proposition
           \ref{PRm} is difficult  to  recognize.

           We 
           give a simple expression of the invariant metric on $\mc{X}^J_n$  without the
           invariant one-forms, using the metric determined on the
           Sigel-Jacobi upper half space $\mc{X}^J_n$  \cite{multi,coord,SB15}.

            \subsection{Invariant metrics on $\mc{X}^J_n$ and $\tilde{\mc{X}}^J_n$}
  Below  $k$, $2k\in \N$ indexes  the holomorphic discrete series of $\Sp$
  and $\nu>0$ indexes 
  the representations of the Heisenberg group. We reformulate for 
 $G^J_n(\R)$, $ n\in\N$, 
i\cite[  Proposition 1]{SB20}  for $G^J_1(\R)$. The starting point is 
   \cite[Proposition 3]{coord}, see also \cite[Theorem 3.2]{SB15}.  
  \begin{Proposition}\label{L11}
a) The \Ka~two-form
\begin{subequations}
  \begin{align*}
      -\ii \omega_{\mc{D}^J_n} (W,z)  & \!=\!
                         \frac{k}{2}\tr(B\wedge\bar{B})\!+\!\nu\tr(A^t\bar{M}\wedge\bar{A}), ~
 A (W,z) \!:=\!\dd z^t\!+\!\dd W\bar{\eta}, W\in\mc{D}_n,\\
B(W)  & := M\dd W, ~M  \!:=\!(\un
            -W\bar{W})^{-1},~z\in
            M(1,n,\C), ~\eta\in M(n,1,\C),
\end{align*}
\end{subequations}
is $(G^J_n)_0$ invariant to the action ${\rm{Sp}}(n,\R)_{\C}\times\C^n:~~  (W,z^t)\rightarrow(W_1,z^t_1)$
\begin{equation}\label{54}
  (\left(\begin{array}{cc}\mc{P} &
                                   \mc{Q}\\\bar{\mc{Q}}&\bar{\mc{P}}\end{array}\right),\!\alpha)
\times (W,z^t)\!=\!((W\mc{Q}^{\dagger}\!+\!\mc{P}^{\dagger})^{-1}(\mc{Q}^t\!+\!W\mc{P}^t),(W\mc{Q}^{\dagger}\!+\!\mc{P}^{\dagger})^{-1}
(z^t\!+\!\alpha^t\!-\!W\alpha^{\dagger})),
\end{equation}
where \eqref{33a} are verified, i.e. 
\begin{equation}\label{MCPQ} 
\mc{P}\mc{P}^{\dagger}-\mc{Q}\mc{Q}^{\dagger}=\un,\quad
                               \mc{P}\mc{Q}^t=\mc{Q}\mc{P}^t,\quad
\mc{P}^{\dagger}\mc{P}-\mc{Q}^t\bar{\mc{Q}}  =\un,\quad \mc{P}^t\bar{\mc{Q}}=\mc{Q}^{\dagger}\mc{P}.
\end{equation} 
We have the change of variables $(W,z)\rightarrow (W,\eta)$
\begin{equation}\label{FC}
FC: ~~  z^t=\eta-W\bar{\eta}; ~~FC^{-1}: ~~ \eta=M(z^t+Wz^{\dagger}), 
\end{equation}
and
\[
  A(W,z)\rightarrow A(W,\eta)= \dd \eta-W\dd \bar{\eta}.
  \]
The complex two-form
\[
  \omega_{\mc{D}^J_n}(W,\eta):=FC^*(\omega_{\mc{D}^J_n}(W,z))
  \]
is not a \Ka~two-form.

The symplectic two-form $\omega_{\mc{D}^J_n}(W,\eta)$ is invariant to
the action $(g,\alpha)\times (W,\eta)\rightarrow (W_1,\eta_1)$ of $(G^J_n)_0$ on
$\mc{D}_n\times\C^n$
\[
\eta_1^t=\mc{P}(\eta+\alpha)^t+\mc{Q}(\eta+\alpha)^{\dagger},
\]
where  $W_1$ is defined in \eqref{54} and $(\mc{P},\mc{Q})$ verify \eqref{MCPQ}.

b) Using the partial Cayley transform
\begin{subequations}
\begin{align}
\Phi^{-1} &: v=\ii (\un-W)^{-1}(\un+W); ~ u^t=(\un-W)^{-1}z^t, ~
              W\in\mc{D}_n,~ v\in \mc{X}_n;\\
\Phi &: W=(v-\ii \un)^{-1}(v+\ii \un), ~  z^t=2\ii(v+\ii \un)^{-1}u^t,  ~z,u\in M(1,n,\C),
\end{align} 
\end{subequations}
we obtain\begin{equation}\label{vvd}
A(W,z)=2\ii (v+\ii\un)^{-1}G(v,u),\quad G(v,u)=\dd u^t-\dd
v(v-\bar{v})^{-1}(u-\bar{u})^t.
\end{equation}

The \Ka~ two-form on $\mc{X}^J_n$ depending on two parameters,
invariant to the action \eqref{exac}  of $G^J_n(\R)_0$, is
\[
      -\ii \omega_{\mc{X}^J_n} (v,u) =
                         \frac{k}{2}\tr(H\wedge\bar{H})+\frac{2\nu}{\ii}\tr(G^tD\wedge\bar{G}),\quad
                         D :=(\bar{v}-v)^{-1},~H:=D\dd v.
                         \]
 
We have the change of variables $FC_1: (v,\eta)\rightarrow (v,u)  $,
where 
\begin{subequations}\label{XCX}
\begin{align}
\eta & =(\bar{v}-\ii \un)D(v-\ii \un)[(v-\ii
       \un)^{-1}u^t-(\bar{v}-\ii\un)^{-1}u^{\dagger}],\\
u^t&=\frac{1}{2\ii}[(v+\ii \un)\eta-(v-\ii \un)\bar{\eta}].
\end{align}
\end{subequations}

c) If we make the change of variables \eqref{uvpq}, then \eqref{vvd} becomes
\[
  G^t(v,u)=\dd u-p\dd v,
  \]
  and 
  \[
    G^t(v,u)=G^t(x,y,p,q)=\dd p v+\dd q= \dd p(x+\ii y)+\dd q.
    \]

d) With \eqref{XCX}, \eqref{uvpq} and 
\begin{equation}\label{215}
M(1,n,\C)\ni\eta:=\chi+\ii \psi, \chi,\psi \in M(1,n,\R),
\end{equation}
 we have the change of coordinates 
\[
  (x,y,p,q)\rightarrow (x,y,\chi,\psi),\quad p^t=\psi, q^t=\chi,
  \]
and 
\begin{equation}\label{217}
G^t(v,\eta)=G^t(x,y,\chi,\psi)=\dd \psi^t x+ \dd \chi^t +\ii \dd
\psi^t y.
\end{equation} 
We obtain  
\[
\eta =  (q+\ii p)^t;\quad
q^t=\frac{1}{2}(\eta+\bar{\eta}),\quad p^t=\frac{1}{2\ii}(\eta-\bar{\eta}).
\]
Given the change if variables \eqref{uvpq} and 
\[
  u:=\xi+\ii \rho,
  \]
 we have the change of variables
\[
  (x,y,\xi,\rho)\rightarrow (x,y,p,q),\quad \xi=px +q, \quad \rho= py,
  \]
and \eqref{217} becomes 
\[
G^t(v,u)=G^t(x,y,\xi,\rho)=\dd \xi-\rho y^{-1}\dd x +\ii (\dd
\rho-\rho y^{-1}\dd y).
\]
With \eqref{uvpq}, \eqref{215} and  \eqref{XCX}, we have the
change of coordinates
\[
(x,y,\xi,\rho)\rightarrow (x,y,\chi,\rho), \quad \xi=\psi^t x+ \chi^t,
\rho=\psi^t y.
\]
\end{Proposition}

We recall that in Perelomov's approach to CS it is considered the triplet
$(G,\pi,\got{H})$, where $\pi$ is a unitary, irreducible representation
of the Lie
group $G$ on the separable complex  Hilbert space $\got{H}$  \cite{perG}. 

We can introduce the normalized (un-normalized) CS-vector
$\underline{e}_x$ (respectively, $e_z$) defined in $z\in M=G/H$
 \begin{equation}\label{2.1}
\underline{e}_x=\exp(\sum_{\phi\in\Delta^+}x_{\phi}{\mb{X}}^+_{\phi}-{\bar{x}}_{\phi}{\mb{X}}^-_{\phi})e_0,
\quad e_z=\exp(\sum_{\phi\in\Delta^+}z_{\phi}{\mb{X}}^+_{\phi})e_0,
\end{equation}
where $e_0$ is the extremal weight vector of the representation $\pi$,
$\Delta^+$ denotes the set  of positive roots
of the Lie algebra $\got{g}$ of $G$, and $X_{\phi}$, $\phi\in\Delta$
are the generators.
$X^+_{\phi}$   ($X^-_{\phi}$) 
 corresponds to
the positive (respectively, negative) generators. See details in \cite{SB03,SB14,perG}.

Let us denote by $FC$ the change of variables $x \rightarrow  z$ in
formula \eqref{2.1} such that
\[
  \underline{e}_x =  (e_z,e_z)^{-\frac{1}{2}}e_z; \quad z = FC(x).
  \]
\cite[Lemma 2]{SB14} verifies the assertion above for CS
defined on $\mc{D}^J_1$,
see also   \cite[Lemma 3]{SB05}, \cite[Lemma 6.11 and  Remark 6.12]{holl}. But the same assertions are true for CS
defined on $\mc{D}^J_n$,   see  \cite[Lemma 7 and Comment
8]{SB06} and 
\cite[Lemma 3.6 and Remark 3.7]{multi}.

Next remark generalizes
\cite[Remark 1]{SB20} established on $G^J_1(\R)$  to
$G^J_n(\R)$, $n\in\N$. 
\begin{Remark}\label{REM11} The $FC$-transform \eqref{FC} relates the  un-normalized CS-vector $e_{Wz}$ to
  the normalized one $\underline{e}_{W\eta}$
\[
\underline{e}_{W\eta}=(e_{Wz}, e_{Wz})^{-\frac{1}{2}}e_{Wz},\quad
W\in\mc{D}_n,~~~z,~\eta^t\in M(1,n,\C),
\]
and the $S_n$-variables $p,q$ are related to parameter $\eta$ defined in
\eqref{FC} by  the relation
\[
\eta=(q+\ii p)^t.
\]

\end{Remark}
  
   If we denote
    $\alpha:=\frac{k}{4}$, $\gamma=:\nu$ and take into consideration
    assertion d) in Lemma \ref{actM}, it is obtained 

    \begin{Theorem}\label{PR2}
      The  metric on $\mc{X}^J_n$, $G^J_n(\R)_0$-invariant
    to the action in {\emph{Lemma \ref{actM}}}, has the expressions
\begin{subequations}\label{MULTL}\begin{align}
      \dd s_{\mc{X}^J_n}^2(x,y,p,q)&\!=\!\alpha\tr[(y^{-1}\dd
                                      x)^2+(y^{-1}\dd y)^2]\nonumber\\  
        &\!+\!\gamma[\dd p
        (xy^{-1}x+yy^{-1}y)\dd p^t+\dd q y^{-1}\dd q^t +2\dd p
        xy^{-1} \dd q^t];\\
  \dd s_{\mc{X}^J_n}^2(x,y,\chi,\psi)&\!=\!\alpha\tr[(y^{-1}\dd
                                      x)^2+(y^{-1}\dd y)^2]\nonumber\\  
        &\!+\!\gamma[\dd \psi^t
        (xy^{-1}x\!+\!yy^{-1}y)\dd\psi\!+\!\dd\chi^ty^{-1}\dd\chi\!+\!2\dd\psi^t
        xy^{-1} \dd\chi];\\
  \dd s_{\mc{X}^J_n}^2(x,y,\xi,\rho)&\!=\!\alpha\tr[(y^{-1}\dd
                                      x)^2+(y^{-1}\dd y)^2]\nonumber \\
  &\!+\!\gamma[ \dd \xi y^{-1}\dd \xi^t +\dd \rho y^{-1}\dd \rho^t
+\rho y^{-1}\dd x y^{-1}(\rho y^{-1}\dd x)^t  \nonumber\\&\!+\!\rho y^{-1}\dd y
    y^{-1}(\rho y^{-1}\dd y)^t  
 -2\rho y^{-1}\dd xy^{-1}\dd \xi^t -2\rho y^{-1}\dd y^{-1}\dd \rho^t]
      \end{align}
      \end{subequations}

      The  three parameter metric on $\tilde{\mc{X}}^J_n$, $G^J_n(\R)$-invariant
      to the action {\emph{c)}} in {\emph{Lemma \ref{actM}}},
      is \begin{equation}\label{BIGM}\begin{split}
        \dd s_{\tilde{\mc{X}}^J_n}^2(x,y,p,q,\kappa)&=
        \dd s_{\mc{X}^J_n}^2(x,y,p,q)+ \lambda_6^2\\
         & =\alpha\tr[(y^{-1}\dd
                                      x)^2+(y^{-1}\dd y)^2]\\  
        & + \gamma[\dd p
        (xy^{-1}x+yy^{-1}y)\dd p^t+\dd q y^{-1}\dd q^t +2\dd p
        xy^{-1} \dd q^t]\\
        & +\delta (\dd \kappa -p \dd q^t+q \dd p^t)^2.
      \end{split}
      \end{equation}
      \end{Theorem}
   
        Formula \eqref{MULTL} (\eqref{BIGM}) is a generalisation to
        $\mc{X}^J_n$ ($\tilde{\mc{X}}^J_n$), $n\in \N$, of
        equation (5.25b) (respectively, (5.30)) in \cite{SB19} corresponding to $n=1$.

           \section{Appendix: Other representations of the Jacobi  algebra}\label{AP7}
We  remind that 
the Jacobi algebra $\got{g}^J_n$, also denoted $ \got{st}( n,\R )$
 by Kirillov
 in \cite[\S 18.4]{kir} or $\got{tsp}(2n+2,\R )$ in \cite{kir2},
 is isomorphic with the subalgebra of Weyl algebra  $A_n$ (see also
\cite{dix}) of polynomials of degree maximum 2 in the variables
$p_1,\dots,p_n, q_1,\dots,q_n$.

           In  \cite{multi} we have considered complex  and biboson
             realization of Lie algebra $\got{sp}(n,\R)$ as \- $\got{sp}(n,\R)_{\C}$
             \[
\got{sp}(n,\R)_{\C}= \Big\langle \sum_{i,j=1}^n \big(2a_{ij}K^0_{ij} +
b_{ij}K^+_{ij}-\bar{b}_{ij}K^-_{ij}\big)\Big\rangle,
\]
where matrices $a=(a)_{ij}$, $b=(b)_{ij}$, $i, j, =1,\dots,n$ verify conditions $a^{\dagger}=-a$, $b^t=b$.
The realization of the generators of the symplectic group in biboson
operators was observed firstly in \cite{gosh}, \cite{mlo}.

The correspondence between the generators \eqref{alg2} and the
generators in \cite[p. 248]{yang} of $G^J_n(\R)$ is
$$2H\rightarrow A+S,~4F\rightarrow B+T,~4G\rightarrow B-T,~D^0
\rightarrow R,~D_{1q}\rightarrow P_q,~\hat{D}_{1q}\rightarrow Q_q.$$

The algebra  $\rm{wsp}(2N,\R)$  in \cite{CQ}, the  semidirect
product of  $\got{sp}(n,\R)$ and Heisenberg,  is essentially the algebra in \cite{multi}, except a factor 2.

The algebra of the inhomogeneous symplectic group $\text{ISp}(2,\R)$
in \cite{kr} is the same  as our Jacobi algebra
$\got{g}^J_1$ in \cite{holl}. 

The Jacobi algebra $\got{g}^J_n$ of the Jacobi group $G^J_n$  is realized
as two-foton algebra \cite{zhang} and $G^J_n$ is embedded in
$\SPn_{\C}$ in the context of mean-field theory in  Nuclear Physics \cite{nish}.

\section{Appendix: Differential of square root of a symmetric
  matrix}\label{AP8}

Let us consider  a  matrix $A\in M(n,\R)$ with the  eigenvalues
$\lambda_1,\dots,\lambda_n$, and let $\R\in\alpha>0$. Then there  exists
a unitary matrix $U$  such that
$$A^{\alpha}=U \mr{dg}(\lambda_1^{\alpha},\dots,\lambda_n^{\alpha}
)U^{\dagger}.$$
For $\alpha= 1/2$, i.e. $A^{1/2}A^{1/2}=A$,  we have
\begin{equation}\label{doua*}
  \dd A^{1/2} A^{1/2} +A^{1/2}\dd A^{1/2}= \dd A.
\end{equation} \eqref{doua*} is a particular case of the  matrix Sylvester equation
\begin{equation}\label{una*}
  AX+XB=C,
\end{equation}
where $A\in M(n),B\in M(m)$ and $X,C\in M(m,n)$. Then the solution $X$
of
the matrix  equation \eqref{una*} can be written as 
\cite{der}
\begin{equation}\label{83}
  (\db{1}_m\otimes A+B^t\otimes \db{1}_n)\mr{vec}(X)=\mr{vec}(C),
\end{equation}
and the solution of the differential equation \eqref{doua*} becomes
\begin{equation}\label{84}
  \mr{vec}(\dd A^{1/2})=((A^t)^{1/2}\oplus
  A^{1/2})^{-1}\mr{vec}(\dd A).
 \end{equation}
$\otimes$ denotes in \eqref{83}, the Kroneker product, $\oplus$ in
\eqref{84} denotes the Kronecker sum, while $\mr{vec}(X)$ denotes the
vectorization of the matrix $X$ \cite{der},\cite{lutke}.

If the matrix $A$ is symmetric and positive definite, we introduce the
notation \cite{der}
$$\alpha:=\mr{vech} (A)=L_nA, \quad a:=D_n \alpha$$
$$\sigma:= \mr{vech}({A}^{1/2}) \quad \sigma = L_nA^{1/2},\quad
A^{1/2}=D_n \sigma,$$
where $\mr{vech}(X)$ denotes the half-vectorization of the matrix $X$,
while 
$D_n$ and $L_n$ denotes the duplication, respectively elimination
matrix, see  \cite{lutke} for definitions.

It is obtained
$$\dd \sigma=L_n ((A^t)^{1/2}\oplus A^{1/2})^{-1} D_n\dd \alpha,\quad
\frac{\pa \sigma}{\pa \alpha}=L_n ((A^t)^{1/2}\oplus A^{1/2})^{-1} D_n.$$

In our case of \eqref{LFGHL}, due to \eqref{c0a01}, we have to replace
for the symmetric  positive definite matrix
$y\rightarrow{y}^{1/2}$, and 
formula \eqref{84} reads
\[
\begin{split}
  \mr{vec} (\dd y^{1/2}) & =(y^{1/2}\oplus y^{1/2})^{-1}\mr{vec}(\dd
  y)=(y^{1/2}\otimes\un + \un\otimes y^{1/2})^{-1}\mr{vec}(\dd y).\\
  \mr{vech} (\dd y^{1/2}) & \!=\!L_n(y^{1/2}\!\oplus \!y^{1/2})^{-1}D_n\mr{vech}(\dd
  y)\\
  &=\!L_n(y^{1/2}\otimes\un \!+\! \un\otimes y^{1/2})^{-1}D_n\mr{vech}(\dd y).
\end{split}
\]

\subsection*{Acknowledgements} 
This research  was conducted in  the  framework of the 
ANCS project  program   PN 19 06
01 01/2019. 
I am indebted to  Professor Arkadiusz Jadczyk for conversations about the  symplectic
group in the 2014. I would like to thank Professor Jae-Hyun  Yang for
sending me his collected papers.


\begin{thebibliography}{99}

\bibitem{ali}S.T. Ali, J.-P. Antoine,  J.-P. Gazeau, {\it Coherent states,
wavelets, and their generalizations}, Springer-Verlag, New York, 2000


\bibitem{arw} B. Arvind,  N. Dutta, Mukunda,  R. Simon, {\it The real symplectic groups in
quantum mechanics and optics}, Pramana - J  Phys {\bf 45} (1995) 
471--497,  arXiv:quant-ph/9509002
 
\bibitem{SB20} Elena Mirela Babalic, S. Berceanu, {\it  Remarks on the geometry of
the extended Siegel--Jacobi upper half-plane}, arXiv:2002.04452 

\bibitem{SB03} S.  Berceanu, Realization of coherent state algebras by
  differential operators, in {\it Advances in Operator Algebras and
    Mathematical Physics},  Editors  F. Boca, O. Bratteli, R. Longo, H. Siedentop, The Theta Foundation, Bucharest 2005, 1--24, 	arXiv:math/0504053 

\bibitem{SB05}  S. Berceanu, {\it A holomorphic representation of Lie algebras
semidirect sum of semisimple and Heisenberg algebras},  Romanian
J. Phys. {\bf 50} (2005) 81--94 


  \bibitem{SB06} S. Berceanu, {\it A holomorphic representation of the
      semidirect sum of symplectic and Heisenberg Lie algebras},
    J. Geom. Symmetry Phys. {\bf 5} (2006) 5--13


  \bibitem{holl}S. Berceanu, {\it A holomorphic representation of the Jacobi algebra},
Rev. Math. Phys.  {\bf 18}  (2006) 163--199; {\it  Errata},   Rev.  Math. Phys.
{\bf 24}  (2012) 1292001, 2~pages, arXiv:math.DG/0408219
 

\bibitem{multi}S. Berceanu,   A holomorphic 
representation of Jacobi algebra in several 
dimensions, in {\it  Perspectives in Operator Algebra and Mathematical
Physics}, Editors   F.-P. Boca, R. Purice,  S. Stratila, \textit{Theta
Ser. Adv. Math.}, Vol.~8, Theta, Bucharest, 2008, 1--25, arXiv:math.DG/0604381


\bibitem{SB81} S. Berceanu, Generalized squeezed states for the Jacobi
  group, in
   {\it Geometric Methods in Physics},
{\textit AIP Conference Proceedings} Vol. 1079, Editors P. Kielanowski,
A. Odzijewicz, M. Schlichenmaier, Th. Voronov, 2008,  67--75,
arXiv:0812.0717

\bibitem{SB82} S. Berceanu, The Jacobi Group and the Squeezed States -
  Some Comments, in  {\it
    Geometric Methods in Physics}, {\textit AIP Conference Proceedings} Vol. 1191, Editors
 P. Kielanowski, S.T. Ali, A. Odzijewicz, M. Schlichenmaier,
 Th. Voronov,  2009,  21--29, arXiv:0910.5563



\bibitem{coord}S. Berceanu,  {\it  A convenient coordinatization of Siegel--Jacobi
    domains},   Rev.  Math. Phys.  {\bf 24}  (2012) 1250024, 38 pages,
  arXiv:1204.5610


\bibitem{FC} S. Berceanu,  {\it Consequences of the fundamental conjecture for the motion
  on the  Siegel--Jacobi disk},   Int.  J.  Geom. Methods Mod. Phys. {\bf
  10} (2013) 1250076, 18 pages,  arXiv:1110.5469


  

 \bibitem{SB14}S. Berceanu,  {\it Coherent states and geometry on the Siegel--Jacobi
disk}, Int. J. Geom. Methods Mod. Phys.  {\bf{11} }
(2014) 1450035, 25 pages,  arXiv:1307.4219

\bibitem{SB15}S. Berceanu, {\it Balanced metric and Berezin quantization on  the
Siegel-Jacobi ball}, SIGMA {\bf 12} (2016) 064, 24 pages,
arXiv:1512.00601

  
\bibitem{SB19}S.  Berceanu, {\it The real Jacobi group
    revisited},  SIGMA {\bf 15} (2019) 096, 50 pages, arXiv:1903.1072 v1, 93 pages;  v2, 54 pages   




\bibitem{gem}S.  Berceanu, A.  Gheorghe,  {\it On the geometry of Siegel-Jacobi domains},
 Int. J. Geom. Methods Mod. Phys. {\bf 8}  (2011) 1783--1798,
 arXiv:1011.3317v1


\bibitem{ber73}F. A. Berezin, {\it Quantization in complex
bounded domains},  Dokl. Akad. Nauk
 SSSR \textbf{211} (1973) 1263--1266


\bibitem{ber74}F. A. Berezin,  {\it  Quantization},
Math. USSR-Izv.  {\bf 38}  (1974) 1116--1175

\bibitem{ber75}F. A. Berezin,  {\it  Quantization in complex symmetric
  spaces},   Math.
 USSR-Izv. {\bf 39} (1975)
363--402

\bibitem{berezin} F. A. Berezin,  {\it The general concept of
    quantization}, 
   Commun. Math. Phys. {\bf 40} (1975) 153--174

    
\bibitem{bern84}R.  Berndt, {\it  Some differential operators in the
    theory of Jacobi forms}, preprint 
IHES/M/84/10,  1984, 31 pages


  
  \bibitem{BB}R. Berndt, S. B\"orcherer, {\it Jacobi Forms and Discrete
      Series Representations of the Jacobi group},  Math. Z. {\bf 204}
      (1990) 13--44


  \bibitem{bs}R. Berndt,  R. Schmidt,  {\it Elements of the representation
theory of the Jacobi group}, Progress in Mathematics,  Vol.  163,  Birkh\"auser
Verlag, Basel, 1998




\bibitem{ca1} B.  Cahen, {\it Global parametrization of scalar holomorphic coadjoint orbits of a
 quasi-{H}ermitian {L}ie group}  Acta Univ. Palack. Olomuc. Fac. Rerum
 Natur. Math. \textbf{52} (2013) 35--48

\bibitem{ca2} B. Cahen, {\it Stratonovich--{W}eyl correspondence for the {J}acobi group},
 Commun. Math.  \textbf{22} (2014)  31--48 

\bibitem{Cah}M. Cahen, S. Gutt, J.  Rawnsley, {\it Quantization of
    K\"ahler manifolds I: geometric interpretation of Berezin's
    quantization},   J. Geom. Phys. {\bf 7} (1990) 45--62

\bibitem{cah}M.  Cahen, S.  Gutt, J.  Rawnsley, {\it Quantization of K\"ahler
  manifolds. II},  Trans. Math. Soc. {\bf 337}
  (1993) 73--98

\bibitem{cart4}\'E. Cartan, 
 {\it La m\'ethode du rep\'ere mobile, la th\'eorie des groupes
   continus et les espaces g\'en\'eralis\'es}, Actualit\'es scientifiques et
 industrielles, Vol. 194, Hermann \& Cie., Paris, 1935

\bibitem{cart5}\'E. Cartan, {\it  Les espaces \`a connexion projective}, \textit{Abh. Sem. Vektor
 - Tensor analysis, Moskau} \textbf{4} (1937) 147--173



\bibitem{der}{\it Derivative (or differential) of symmetric square root of a matrix}, Mathematics Stack Exchange:
https://math.stackexchange.com/questions/540361/derivative-or-differential-of-symmetric-square-root-of-a-matrix


\bibitem{dix} J. Dixmier, {\it Sur les alg\'ebres de Weyl},
Bull. Soc. Math. France {\bf 96} (1968) 209--242 


\bibitem{dod}V.V. Dodonov, I.A. Malkin, V.I. Man'ko, {\it Even
  and odd coherent states and excitations of a singular oscillator},
Physica {\bf 72} (1974) 597--615


\bibitem{don}S.  Donaldson,  {\it Scalar curvature and projective
    embeddings, I},   J. Diff. Geom. {\bf 59}   (2001) 479--522  

  
\bibitem{dr} P.D.  Drummond, Z. Ficek, Editors, \emph{Quantum Squeezing}, Springer, Berlin, 2004

\bibitem{ez}M. Eichler,  D. Zagier,  {\it The theory of Jacobi forms},
 Progress in
 Mathematics, Vol.~55, Birkh\"{a}user Boston, Inc., Boston, MA, 1985


\bibitem{eng}M.
Engli{\v{s}}, {\it Berezin quantization and reproducing kernels on complex
 domains}, Trans. Amer. Math. Soc. \textbf{348} (1996) 411--479
 
\bibitem{ev} E.L. Evtushik (originator)  Moving-frame method, in {\it
    Encyclopedia of Mathematics}.  \newline
http://www.encyclopediaofmath.org/index.php?title=Moving-frame-method\&oldid=17828

\bibitem{frei}  E. Freitag, {\it Siegelsche Modulfunktionen}, Springer Verlag, 1983


\bibitem{pedro}P.J.  Freitas,  {\it On the action of
the symplectic group
on the
Siegel upper half plane}, Ph.D. Thesis,
University of Illinois at Chicago, 1999

  
\bibitem{fol} G.B. Folland, \emph{Harmonic analysis in phase space},
\rm{Annals of Mathematics Studies, Vol. 122, Princeton Univ.
Press, Princeton, NJ, 1989}


\bibitem{gosh}S. Goshen, H.J. Lipkin, {\it  A simple independent-particle
  system having collective properties},   Ann. Phys. (N. Y.)   {\bf 6}  (1959) 301--309


\bibitem{goss}M. de Gosson, {\it Symplectic Geometry in Quantum
    Mechnaics}, Birkh\"auser, Basel, 2006

\bibitem{hagen}C.R.  Hagen, {\it Scale and conformal transformations in
  Galilean-covariant field theory},   Phys. Rev. D 
{\bf 5}   (1972) 377--388


\bibitem{helg} S. Helgason, \emph{Differential geometry, Lie groups and
symmetric spaces}, \rm{Academic Press, New York, 1978}



\bibitem{ho}J.N. Hollenhors, {\it Quantum limits on resonant-mass
gravitational-wave detectors}, {Phys. Rev. D} {\bf 19} (1979) 1669--1679



\bibitem{cal3} E.  K\"ahler,
{\it Raum-Zeit-Individuum},
Rend. Accad. Naz. Sci. XL Mem. Mat. {\bf{16}}  (1992)
115--177 


\bibitem{ken} E.H. Kennard, {\it{Zur Quantenmechanik einfacher
Bewegungstypen,}} {Zeit. Phys.} \textbf{44} (1927)
326--352



\bibitem{kir}A. Kirillov, {\it \'El\'ements de la th\'eorie des
repr\'esentations}, Editions Mir, Moscou, 1974

\bibitem{kir2}A.A. Kirillov, {\it Merits and demerits of the orbit method},
Bull. Amer. Math. Soc. {\bf 36} (1999) 43-73


\bibitem{kir1}A.A. Kirillov, {\it Lectures on the orbit method}, Graduate
studies in Mathematics, Vol. 64, American Mathematical Society, Providence,
Rhode Island, 2004 



\bibitem{hua}{L.K. Hua,}
{\it Harmonic Analysis of Functions of Several Complex Variables
in the Classical Domains}, Amer. Math. Soc., Providence, R.I., 1963



\bibitem{kr}P. Kramer, M. Saraceno, {\it Semicoherent states and the
    group} $\text{ISp}(2,\R)$, {Physics} \textbf{114A} (\rm{1982}) 448--453


\bibitem{kn1} S. Kobayashi, K. Nomizu, {\it Foundations of differential
geometry}, {V}ol.~{I},
 Interscience Publishers, New York~-- London, 1963

\bibitem{kn} S. Kobayashi, K. Nomizu, {\it Foundations of differential
geometry}, {V}ol.~{II}, Interscience Publishers, New York~-- London~-- Sydney, 1969

\bibitem{LEE03} M.H. Lee, \textit{Theta functions on hermitian symmetric
domains and Fock representations}, J. Aust. Math. Soc. \textbf{74} (2003) 201--234

 

\bibitem{lis2}W. Lisiecki, {\it A classification of coherent state 
representations of  unimodular Lie groups}, Bull. Amer. Math. Soc. 
 {\bf 25} (1991) 37--43

 \bibitem{lis}W. Lisiecki, {\it Coherent state representations. A survey},
  Rep. Math. Phys. {\bf 35} (1995) 327--358

  

\bibitem{lu}E.Y.C. Lu,  {\it New coherent states of the
electromagnetic field}, Lett. Nuovo. Cimento  {\bf 2}  (1971)
1241--1244 
  
\bibitem{lutke}H. L\"utkepohl, {\it Handbook of matrices},  John Wiley
  \& Sons,  Chikester, 1996


\bibitem{mandel} L. Mandel,  E. Wolf, {\it Optical coherence and quantum optics}, Cambridge University Press, 1995



\bibitem{mlo}L.D. Mlodinow, N. Papanicolau, $SO(2,1)$ {\it algebra and the
  large N expansion in quantum mechanics},  Ann.  Phys. (N. Y.) {\bf 128} (1980) 314--334


\bibitem{mo}B.R. Mollow, R.J. Glauber, {\it Quantum theory of
parametric amplifications: I}, Phys. Rev. {\bf 160} (1967) 1076-1096
   
  
  
\bibitem{mosc}H. Moscovici, {\it Coherent state representations of
nilpotent Lie groups},  Commun. Math. Phys. {\bf 54} (1977)  63--68

\bibitem{mv}H. Moscovici,  A. Verona, {\it Coherent states and square integrable
representations},  Ann. Inst. Henri Poincar\'e {\bf 29}
(1978) 139--156




\bibitem{neeb96}K.-H. Neeb, {\it Coherent states, holomorphic
    extensions
    and highest weight representations},  Pacific J. Math. {\bf 174}
(1996) 230--261


\bibitem{neeb}K.-H. Neeb, {\it Holomorphy and convexity in Lie
    theory},   De Gruyter
 Expositions in Mathematics, Vol.~28, Walter de Gruyter \& Co., Berlin, 2000

 
\bibitem{ni}U. Niederer, {\it  The maximal kinematical invariance
group of the harmonic oscillator}, {Helv. Phys. Acta} {\bf 46} (1973)
191--200



\bibitem{nish}S. Nishiyama, J. Da Providencia, {\it Mean-field theory
    based on the $\got{Jacobi~hsp}$:= semidirect sum $\got{h}_N \rtimes
    \got{sp}(2N,\R)_{\C} $ algebra of boson operators},
  J. Math. Phys. {\bf{60}} (2019) 081706, 22 pages,  	arXiv:1809.01314 



\bibitem{perG}A.M.   Perelomov,  {\it Generalized coherent states and their
applications}, Texts and
 Monographs in Physics, Springer-Verlag, Berlin, 1986


\bibitem{CQ} C. Quesne, {\it Vector coherent state theory of the
    semidirect sum Lie algebras $\rm{wsp}(2N,\R)$}, J. Phys. A:
  Gen. {\bf 23}  (\rm{1990}) 847--862

  
\bibitem{raw}  J.H. Rawnsley,  {\it Coherent states and K\"ahler
    manifolds},  Quart. J. Math. Oxford Ser. {\bf 28}  (1977) 403--415



\bibitem{SA71} I. Satake, \textit{Fock representations and Theta Functions},
  Ann. Math. Studies \textbf{66} (1971) 393--405

\bibitem{SA71B} I. Satake, \textit{Unitary representations of a semi-direct
products of Lie groups on }$\bar{\partial}$\textit{-cohomology spaces},
Math. Ann.  \textbf{190} (1971) 177--202

\bibitem{SA71C} I. Satake, \textit{Factors of automorphy and Fock
representations}, Advances in Math. \textbf{7} (1971) 83-110 

\bibitem{SA80} I. Satake, \textit{Algebraic structures of symmetric
    domains}, Kan\^o Memorial Lectures, 4,  Iwanami Shoten, Tokyo; Princeton University Press, Princeton, N.J., ,  1980


  
\bibitem{ser} D. Serre, {\it Matrices: theory and applications}, Springer,
  New York, 2002


  
\bibitem{sieg}C.L. Siegel, {\it Symplectic geometry},  Academic
  Press, New York, 1964
  
\bibitem{si} R. Simon, E.C.G. Sudarshan, N. Mukunda, {\it Gaussian
pure states in quantum mechanics and the symplectic group},
Phys. Rev. A\, \textbf{37} (1988)  3028--3038


\bibitem{siv} S. Sivakumar, {\it Studies on nonlinear quantum optics}, { J. Opt. B Quantum Semiclass. Opt.}
  \textbf{2}  (2000)  R61--R75


\bibitem{stol}P. Stoler, {\it Equivalence classes of minimum uncertainty packets},
 {Phys. Rev. D} {\bf  1} (1970) 3217--3219


\bibitem{tak}K. Takase, {\it  A note on automorphic forms}, J. Reine
  Angew. Math. {\bf 409} (1990) 138--171
  
\bibitem{tam}T.-Y. Tam, {\it Computing the Iwasawa decomposition of a
    symplectic matrix by Cholesky factorization}
  Appl.  Math. Lett. {\bf{19}} (2006) 1421-1424

  
\bibitem{terras} A. Terras, {\it Analysis on symmetric spaces and
applications}, II, Springer-Verlag, Berlin, 1988


\bibitem{wolf} J.A. Wolf, Fine structure of Hermitian symmetric
spaces, in \emph{Symmetric spaces} (Washington Univ., St. Louis,
Mo., 1969-1970), Editors  W.M. Boothby, G.L. Weiss,  Marcel
Decker, New York, 1972,  271--357



\bibitem{kbw1} K.B.  Wolf, The Heisenberg-Weyl ring in quantum
mechanics, in
 {\it Group theory and its applications\/},  Vol. 3, Editor  E.M. Loebl,
Academic Press, New York, 189--247, 1975




 \bibitem{Yg93}J.-H. Yang, {\it The Siegel-Jacobi operator}, 
   Abh. Math. Sem. Univ. Hamburg {\bf 63} (1993) 135--146
   
\bibitem{yang} J.-H. Yang, \textit{The method of orbits for real Lie groups},
Kyungpook Math. J. {\bf 42}  (2002) 199--272, arXiv:math.RT/060205


\bibitem{Y07} J.-H. Yang, \textit{Invariant metrics and Laplacians on the
    Siegel--Jacobi spaces}, J.  Number Theory, {\bf 127} (2007) 83--102,
  arXiv:math.NT/0507215

\bibitem{Y08} J.-H. Yang, \textit{A partial Cayley transform for
    Siegel--Jacobi disk}, J. Korean Math. Soc.  {\bf 45}  (2008) 781--794,
  arXiv:math.NT/0507216


\bibitem{Y10} J.-H. Yang, \textit{Invariant metrics and Laplacians on the
    Siegel--Jacobi disk}, Chin. Ann. Math.  {\bf 31B}   (2010) 85--100,
  arXiv:math.NT/0507217


\bibitem{yu}H.P. Yuen, {\it  Two-photon coherent states of the
radiation field}, {Phys. Rev. A} {\bf 13} (1976) 2226--2243

  
\bibitem{ZIEG}C. Ziegler, {\it Jacobi Forms of Higher Degree},
  Abh. Math. Sem. Univ. Hamburg {\bf 59} (1989) 191--224


  \bibitem{zhang}W.-M. Zhang, Da-H. Feng, R. Gilmore, {\it Coherent states: theory and some applications},
Rev. Mod. Phys. {\bf 62} (1990) 867--928


  \end{thebibliography}
\end{document}